\documentclass[final,3p,times,12pt]{elsarticlearxiv}
\usepackage[utf8]{inputenc}
\usepackage{hyperref}
\usepackage{verbatim}
\usepackage{mathtools}
\usepackage{color}
\usepackage[normalem]{ulem} % for strike out text by \sout REMOVE LATER
\usepackage{amssymb,amstext,amsmath,mathabx}
\usepackage{amsthm}
\theoremstyle{definition}
\newtheorem{definition}{Definition} 
\newtheorem{remark}{Remark} 
\newtheorem{theorem}{Theorem} 
\newtheorem{corollary}{Corollary} 
 
\newcommand{\no}[1]{\widebar{#1}}

\def\C{\mathcal{C}}

\def\F{\mathcal F}
\def\G{\mathcal{G}}
\def\H{\mathcal H}
\def\I{\mathcal{I}}

\def\M{\mathcal M}
\def\P{\mathcal P}

\def\prev{\mathbb{P}}

\renewcommand\bot{\emptyset}
\journal{This preprint was accepted for publication in IJAR. The final version is available at  \href{http://dx.doi.org/10.1016/j.ijar.2020.03.001}{http://dx.doi.org/10.1016/j.ijar.2020.03.001}}
\begin{document}
\begin{frontmatter}


%\title{
%Probabilities of conditionals and previsions of conjoined and iterated conditionals
%}
\title{
Probabilities of conditionals and previsions of iterated conditionals
}

\author[gs]{Giuseppe Sanfilippo\corref{cor1}\fnref{fn1}}
\address[gs]{Department of Mathematics and Computer Science, University of Palermo, Italy}
\ead{giuseppe.sanfilippo@unipa.it}
\author[ag]{Angelo Gilio\fnref{fn2}}
\address[ag]{Department of Basic and Applied Sciences for Engineering, University of Rome ``La Sapienza'', Italy}
\ead{angelo.gilio1948@gmail.com}

\author[do]{David E.\ Over}
\address[do]{Department of Psychology, Durham University,  United Kingdom}
\ead{david.over@durham.ac.uk}
\author[np]{Niki Pfeifer}
\address[np]{Department of Philosophy, University of Regensburg, Germany}
\ead{niki.pfeifer@ur.de}
\fntext[fn1]{Also affiliated with INdAM-GNAMPA, Italy}
\fntext[fn2]{Retired}
\cortext[cor1]{Corresponding author}

\begin{abstract}
We analyze selected iterated conditionals in the framework of conditional random quantities. 
We point out that it is instructive to examine Lewis’s triviality result, which shows the conditions a conditional must satisfy for its probability to be the conditional probability. In our approach, however, we avoid triviality because  the import-export principle is invalid. 
We then analyze an example of reasoning under partial knowledge where, given a conditional \emph{if $A$ then $C$} as information, the probability of $A$ should intuitively increase. We explain this intuition by making  some implicit background information explicit.
We consider several (generalized) iterated conditionals, which allow us to formalize different kinds of latent information. We verify that for these iterated conditionals the prevision is greater than or equal to the  probability of $A$. We also investigate the lower and upper bounds of the \emph{Affirmation of the Consequent} inference. We conclude our study with some remarks on the supposed ``independence'' of two conditionals, and we interpret this property as \emph{uncorrelation} between two random quantities.
\end{abstract}




\end{frontmatter}

\section{Introduction and motivation}
\begin{quote}
``Even the crows on the roofs  caw about the nature of conditionals'' \\ (Callimachos quoted after \cite{kneale84}, p.\ 128)
\end{quote}
Debates about the nature of conditionals have a very long tradition in philosophy and logic, which go back at least to Diodorus Cronus and his pupil Philo 
(\cite{kneale84}). 
 Even at that early stage, the emerging debate inspired the famous Callimachos epigram. The logical tradition in the study of conditionals has recently led to the popularity of probabilistic approaches. Among the various interpretations of probability
(see, e.g., \cite{fine73}),
we adopt the subjective analysis, which is due to de Finetti 
 (\cite{definetti31,definetti37})
 and  Ramsey (\cite{ramsey94}). 
The probabilistic theory of de Finetti, based on the well-known coherence principle, has been studied and extended  by many authors (see, e.g.,
\cite{Baioletti2019,BCPV11,berti17,BeRR98,biazzo00,biazzo05,CaLS07,cosco96,coletti02,coletti15,flaminio18,gilio02,gilio12ijar,gilio16,gilio13ijar,gilio13ins,lad96,PeVi14,PeVa17}).
We recall that the coherence principle of de Finetti plays a key role in probabilistic reasoning. Indeed, it allows us to extend  any  coherent assessment, on an arbitrary family of (conditional) events, to  further (conditional) events (\emph{fundamental theorem of probability}).
De Finetti and Ramsey have held that the
probability of a natural language conditional, $P($if $A$ then $C)$, and the conditional subjective probability of $C$ given $A$,
$P(C|A)$, are closely related to each other (\cite{bennett03}). Identifying these probabilities has such
far-reaching consequences that it has simply been called \emph{the Equation} (\cite{edgington95}), but we will refer to it here
as the \emph{conditional probability hypothesis}:
\begin{quote}
(CPH)\, $P($if $A$ then $C) = P(C|A)$.
\end{quote}
A conditional that satisfies the CPH has been termed a \emph{conditional event} (\cite{definetti37}), and a \emph{probability
conditional} (\cite{adams98,lewis76}). 
 The CPH is usually adopted in conjunction with the Ramsey test (\cite{bennett03,edgington95,ramsey94}), which states that a probability judgment is made about a conditional, \emph{if $A$ then $C$}, by judging the
probability of $C$ under the supposition that $A$ holds, resulting in a degree of belief in \emph{$C$ given $A$}, the conditional
subjective probability $P(C|A)$. We will use \emph{conditional event} for such a conditional, as our work here depends
so much on de Finetti. In our  approach, a natural language conditional \emph{if $A$ then $C$} is looked at as a three-valued logical entity which coincides with the  conditional event $C|A$, formally defined in Section \ref{sect:2.1}. Then, CPH appears as a natural consequence.

One of the important implications of the CPH is that the natural language conditional cannot be identified with the
material conditional of truth functional logic, which is equivalent to $\no{A} \vee C$ (``not-$A$ or $C$'', as defined as usual) and which resembles Philo's ancient conception of conditionals. The material conditional is not a
conditional event, since $P(\no{A} \vee C) = P(C|A)$ only in extreme cases, e.g. when $C=A$, in which case $P(\no{A}\vee  C)=P(\no{A}\vee  A) =1= P(A|A) =P(C|A)$, while in general it holds that $P(C|A) \leq P(\no{A} \vee C)$.
Some philosophers (see, e.g., \cite{grice89}) and psychologists (see, e.g., \cite{johnsonlaird91}) have tried to argue that the natural language conditional is a material conditional
at its semantic core, but this cannot be right if the CPH is true (see \cite{gilio12} for a detailed study of the probability of disjunctions and the conditional probability). 

There are logical and philosophical arguments, and appeals to intuition, in support of the CPH. Consider the simple
indicative conditional:
\begin{quote}
(SH)\, If you spin the coin ($S$), it will land heads ($H$).
\end{quote}
Suppose we believe that the coin is fair. It would then seem clear that the probability of (SH), \emph{$P($if $S$ then $H)$}, is 0.5,
i.e. $P(H|S)=.5$ for a fair coin. If we suspect the coin is biased, we could spin it a large number of times $m$ and
record the number of times $n$ that it came up heads. The ratio $n/m$ would give us direct evidence about $P(H|S)$
for the next spin, and intuitively that would tell us how to judge \emph{$P($if $S$ then $H)$}. We could then make a conditional bet
on the next spin, by placing ``I bet that'' at the beginning of (SH), using $P(H|S)$ as the 
probability that we will win the bet (assuming that the bet is not called off), thus fixing the rational odds for it. Both \cite{definetti37} and \cite{ramsey26} related
conditionals to conditional bets, and it is striking that the two founders of contemporary subjective probability
theory could have independently developed views of the conditional that were so similar to each other.

Influenced by the above points, psychologists of reasoning have tested the CPH in controlled experiments, and found that
people tend to conform to it for a wide range of conditionals, from indicative conditionals (\cite{evans03,fugard11a,pfeifer13b}) and conditional bets (\cite{baratgin13a}) to counterfactuals and causal conditionals (\cite{OverCruz18,over07b,pfeifertulkkicogsci17,pfeifertulkki17}). In the opinion of some authors, the CPH
may not hold for conditionals \emph{if $A$ then $C$} when the antecedent $A$ does not raise the probability of the conditional's consequent~$C$ (\cite{douven16,skovgaard16}), but the significance of this finding is open to dispute (\cite{OverCruz18}), and the
general support for the CPH has had a major impact in psychology (\cite{elqayam13,oaksford07}), formal epistemology (\cite{pfeifer12y}), and philosophical logic (\cite{pfeiferHOR20,PS2017,SUM2018PS,PfSa19}). But
what can we say about the semantics and logic of conditionals, under CPH? 

Stalnaker used a possible worlds analysis to give the formal semantics and logical properties of a
conditional that he claimed satisfied the CPH (\cite{stalnaker68,stalnaker70}),  which is indeed sometimes termed \emph{Stalnaker's hypothesis} (\cite{douven11b}). 
Stalnaker's conditional, \emph{if $A$ then $C$}, is true in a possible world $w_x$ if and only if $C$ is true in the
closest possible world $w_y$ to $w_x$ in which $A$ is true, the closest $A$ world. The closest $A$ world to $w_x$ is determined by
similarity to $w_x$. Where $w_y$ is the closest possible $A$ world to $w_x$, $w_y$ will be the most similar $A$ world to $w_x$ in facts
and scientific laws, with the proviso that $A$ will be true in $w_y$. Every world is most similar to itself, and so when $A$
is true at $w_x$, \emph{if $A$ then $C$} is true at $w_x$ if and only if $A \wedge C$ is true at $w_x$. The Stalnaker conditional \emph{if $A$ then $C$} is
also always true or false  at a possible world in which $A$ is false, a not-$A$ world. For when $A$ is false at $w_x$, \emph{if $A$ then
$C$} will be true at $w_x$ if and only if $A \wedge C$ is true at the closest $A$ world $w_y$ to $w_x$. 

For example, suppose there are  four possible worlds, $A\wedge C$, $A\wedge \no{C}$, $\no{A}\wedge C$, and $\no{A}\wedge \no{C}$. Then, \emph{if $A$ then $C$} will be true in the $A\wedge C$ world, and false in the $A\wedge \no{C}$ world. We assume that all of these worlds have the probability of .25.
Supposing that the $A\wedge C$ world is closer to the $\no{A}$ worlds than the $A \wedge \no{C}$ world, the probability of the Stalnaker conditional \emph{if $A$ then $C$} is obtained by the sum of the probabilities over the three worlds, $P(A \wedge C)+P(\no{A} \wedge C)+P(\no{A} \wedge  \no{C})=.75$. This is because the Stalnaker conditional is then true in the three worlds: $A \wedge C$, $\no{A} \wedge C$, and $\no{A} \wedge  \no{C}$. Now we note that the conditional probability of $C|A$ is .5. But the probability of the Stalnaker conditional is .75, since it is true in three of the four worlds, and it has their probabilities. If, differently from above, we suppose that the
$A \wedge \no{C}$ world is closer to the $\no{A}$ worlds than the $A\wedge C$ world, then the probability of the Stalnaker conditional \emph{if $A$ then $C$} 
is obtained by the probability of one world only, namely
 $P(A \wedge C)=0.25$. This is because the Stalnaker conditional is now true in just the world:  $A \wedge C$.
The above examples show that  the probability of the Stalnaker conditional is in general different from  the conditional probability $P(C|A)$.  

Lewis did more than illustrate this point: he proved a stronger result. 
Lewis strictly proved that the CPH will not generally be satisfied for a conditional like Stalnaker's, or Lewis's own in
\cite{lewis73}. His proof is very instructive in revealing what the semantics of a conditional must be like for it to
satisfy the CPH, and so for it to be a conditional event. There are a number of variations of Lewis's proof, and some
further proofs have been inspired by it (\cite{bennett03,douven11b}),  but we discuss below  a version of the
proof that best illustrates the properties that a conditional must have to satisfy the CPH and so be a conditional
event. Considering the relation between a conditional \emph{if $A$ then $C$} and its consequent $C$, and using the total probability
theorem, in \cite{lewis76} it was derived:
\[P(\text{\em if } A \text{ \em  then } C) = P(C)P((\text{\em if } A \text{ \em  then  } C)|C) + P(\no{C})P((\text{\em if } A \text{ \em  then  }C)|\no{C}).
\]
Assuming the CPH, Lewis then argued that $P(($if $A$ then $C)|C)$ should be 1,\footnote{In logic under the material conditional interpretation,  it is easy to see  that  \emph{if $C$, then (if $A$ then $C$)} is logically true and, by the deduction theorem, the inference of \emph{if $A$ then $C$} from $C$ is logically valid. This inference is also known as one of the paradoxes of the material conditional, which is absurd when instantiated by natural language conditionals. For standard approaches to probability, which define the conditional probability $P(C|A)$ by $P(AC)/P(A)$---where $P(A)>0$ is assumed to avoid a fraction over zero---one obtains $P(C|A)=1$ when $P(C)=1$ (provided $P(A)>0$). Thus, in this case, this paradox of the material conditional is inherited by such standard approaches to probability. In the coherence approach, however, this paradox of the material conditional is blocked since even if $P(C)=1$, $0 \leq P(C|A) \leq 1$  (and it is not assumed that $P(A)>0$; for a proof see \cite{pfeifer13}).} and \emph{$P(($if $A$ then $C)|\no{C})$}
should be 0, with the result that \emph{$P($if $A$ then $C) = P(C)$}. This is an absurd result. As Lewis put it, \emph{$P($if $A$ then $C)$}
can only be $P(C)$ for ``trivial'' probability functions. He therefore inferred, by \emph{reductio ad absurdum}, that CPH must
be rejected. He justified his claim that, given CPH, \emph{$P(($if $A$ then $C)|C)$}  should be 1 by arguing that it should
then equal the result of learning $C$ for sure, denoted by $P_C({\dots})$, yielding $P_C($if $A$ then $C) = P_C(C|A) = 1$, and
similarly for \emph{$P(($if $A$ then $C)|\no{C})$} and learning not-$C$ for sure,  $P_{\no{C}}({\ldots})$, with \emph{$P_{\no{C}}($if $A$ then $C)= P_{\no{C}}(C|A) = 0$} following.

In an early reply, van Fraassen  (\cite{vanfraassen76}) cast doubt on Lewis's presupposition that the semantics of a natural language
conditional is independent of a person's subjective epistemic state (see also \cite{kaufmann09}). The semantics of a
Stalnaker conditional is based on a similarity relation between possible worlds, which is a (kind of subjective) qualitative comparison between 
possible worlds, and this conditional is always objectively true or false, at $A$ worlds and $\no{A}$ worlds, as we saw above. But it
has to be different for a conditional event. A conditional event \emph{if $A$ then $C$} cannot be objectively true or false in
$\no{A}$ worlds. As we will explain below in a more formal analysis, a conditional event $C|A$ is looked at as a
three-valued logical entity, with values true, or false, or void (with an associated subjective degree of belief),
according to whether $C\wedge A$ is true, or $\no{C}\wedge A$ is true, or $\no{A}$ is true, respectively. Moreover, from a numerical
point of view $C|A$ becomes a random quantity (the indicator of $C|A$) with values: 1, when $C|A$ is true;  0, when $C|A$
is false; $P(C|A)$, when $C|A$ is void. Notice that the use of $P(C|A)$ as numerical
representation of the logical value void plays a key role both in theoretical developments and in  algorithms (for
instance, the betting scheme, the penalty criterion, coherence checking and propagation). Then, (by identifying logical and
numerical aspects) the conditional event $C|A$ has the conditional probability, $P(C|A)$, as its semantic
value in $\no{A}$ cases (\cite{cruz14,gilio90,GiSa14,Jeff91,OverCruz18,pfeifer09b}). This value does, of course, depend on subjective mental states, which concern the uncertainty on $C$ (when $A$ is assumed to be true), and the effect of these on conditional probability judgments. The conditional event does not
acquire probability from $\no{A}$ worlds, because it is not objectively true at these worlds. Its value, $P(C|A)$,
in these worlds is its overall expected value, or prevision in de Finetti's terms, across all the $A$-worlds (indeed, the
value $P(C|A)$ is the result of a mental process in which $A$ is assumed to be true and $C$ may be true or false
across the $A$-worlds). The value $P(C|A)$ can be determined in a Ramsey test of \emph{if $A$ then $C$}, by using
operatively  the betting scheme or the penalty criterion of de Finetti.

Another way to justify Lewis's claim about what follows from the CPH is to use what has been called the import-export
principle (\cite{mcgee85}):
\begin{quote}
(IE)\, $P($if $B$ then $($if $A$ then $C)) = P($if $(A \wedge B)$ then $C)$.
\end{quote}
Using IE and a general form of CPH, one can prove that \emph{$P(($if $A$ then $C)|C)$} should be 1, and \emph{$P(($if $A$ then
$C)|\no{C})$} should be 0. Indeed,  assuming CPH and IE, as \emph{$P(($if $A$ then $C)|C) = P($if $C$ then $($if $A$
then $C)$}, we can infer  that \emph{$P($if $C$ then $($if $A$ then $C)) = P($if ($C \wedge A)$ then $C) = P(C|(A \wedge C)) = 1$}. Similarly under
those assumptions, it follows that \emph{$P(($if $A$ then $C)|\no{C}) = 0$}. It is therefore clear that, to avoid
triviality, IE must fail for the conditional event, and  in our formal analysis  IE is false for
this conditional (\cite{GiSa13c,GiSa14}).
 The expected values, or  previsions (denoted by the symbol $\prev$), of \emph{$($if $B$ then $($if $A$ then $C))$} and \emph{$($if
$(A \wedge B)$ then $C)$} can diverge.  Indeed, in our approach, it holds that 
\[
P(C|A)=\prev((C|A)|C)P(C)+\prev((C|A)|\no{C})P(\no{C}),
\]
which in general does not coincide with $P(C)$ because $\prev((C|A)|C)\neq P(C|AC)=1$ and $\prev((C|A)|\no{C})\neq P(C|A\no{C})=0$ (see \cite[Theorem 6]{GiSa14}). Thus Lewis's triviality is avoided (see also \cite{Over2015} for an experimental study of Lewis / Stalnaker conditionals and the import-export principle).

Douven and Dietz (\cite{douven11b}) aimed to show that there is a serious problem with the CPH without making assumptions about the
relation between conditionals and subjective semantic values. Their argument depends on their observation that a
conditional that satisfies the CPH will be probabilistically independent of its antecedent:
\begin{quote}
(IA)\, \emph{$P(A|($if $A$ then $C)) = P(A)$}.
\end{quote}
We will examine the validity of (IA) formally for the iterated conditional $A|(C|A)$ in what follows (see Section~\ref{SEC:AgCgA}), and make some points about probabilistic independence and uncorrelation in our approach (see Section~\ref{SEC:INDINC}). But at this point, we
will focus on the instructive argument of Douven and Dietz that (IA) should be rejected for the natural language
conditional, implying that CPH is false. Douven and Dietz used an example for their argument, and we will slightly simplify the conditional as follows: 
\begin{quote}
(ES)\,  If Sue passed the exam ($A$), she will go on a skiing holiday ($C$). \end{quote}
In the example, Harry sees Sue buying some skiing equipment, and this surprises him because he knows that she recently
had an exam, which he believes she is unlikely to have passed, making $P(A)$ low for him. But then Harry meets Tom, who
tells him that (ES) is likely. This information increases \emph{$P($if $A$ then $C)$} for Harry, which means that $P(C|A)$
increases for Harry, assuming that CPH holds. Douven and Dietz argue that it is intuitively right that $P(A)$ should go up
for Harry, given this information about (ES), since a high $P(A)$ explains why Sue bought the skiing equipment. But then
(IA) cannot be accepted, because it implies that $P(A)$ will be unaffected by conditioning on \emph{if $A$ then $C$}. We draw a different conclusion from this useful example. 

Among other things we will show that, by replacing   the antecedent $C|A$ by $(C|A)\wedge C$, we will reach the conclusion that the degree of belief in $A|((C|A)\wedge C)$ is greater than or equal to $P(A)$. This reasoning is a form of abductive inference (\cite{sep-abduction}, see also \cite{DuGK08}), which is related to the classical ``fallacy'' of \emph{Affirmation of  the Consequent} (AC): from {\em if $A$ then $C$} and $C$ infer $A$. 
Under the material conditional interpretation of conditionals, AC is not logically valid. Therefore, it is classically called a fallacy. However, 
 we will compute the lower and upper bounds for the conclusion of AC, showing when AC can be a strong, though not valid, inference for reasoning under partial knowledge. We will expand our analysis by considering iterated conditionals. We will use this analysis to take account of implicit information that can be present in particular contexts. This added information will explain the intuition that the degree of belief in $A$ should sometimes increase given {\em if $A$ then $C$}. We will also make some comments about uncorrelation and probabilistic independence, pointing out in particular that the event $A$ and the conditional {\em if $A$ then $C$} are uncorrelated, but not probabilistically independent, and we will explain why this distinction does not lead to counterintuitive results.

The outline of the paper is as follows. In Section~\ref{SEC:PRELIM} 
 we recall some basic notions and results on  coherence,  logical operations among conditional events, and iterated conditionals.
In Section~\ref{SEC:DgCgA} we study the antecedent-nested conditional 
\emph{if ($C$ when $A$), then $D$}. 
In Section~\ref{SEC:AgCgA} we analyze the particular case where $D=A$, by studying  the antecedent-nested conditional 
\emph{if ($C$ when $A$), then $A$} and by  considering selected related iterated conditionals.  In addition, we discuss and apply our results to the above mentioned Harry and Sue example (ES). 
In Section~\ref{SEC:AgCgAandC} we consider the generalized iterated conditional $A|((C|A)\wedge C)$ 
which allows 
to make some latent information explicit, such that the degree of belief in $A$ increases.
In Section~\ref{SECT:AC} we refine the study of the lower and upper bounds for the AC rule.
In Section~\ref{SEC:FURTHERITER} we examine selected further cases where, based on suitable generalized iterated conditionals,  the degree of belief in $A$ increases.
In Section~\ref{SEC:INDINC} 
we discuss and correctly interpret a product formula for the conjunction of two conditionals using the notion of uncorrelation, rather than probabilistic independence, between two random quantities.
In Section~\ref{SEC:CONCLUSIONS} we  give a summary of the paper, by adding  some final comments.

\section{Preliminary notions and results}
\label{SEC:PRELIM}
In this section we recall some basic notions and results  concerning coherence (see, e.g., \cite{biazzo05,BiGS12,CaLS07,coletti02,PeVa17}),  logical operations among conditional events, and iterated conditionals (see \cite{GiSa13c,GiSa13a,GiSa14,GiSa17,GiSa19,SPOG18}).

\subsection{Conditional events and coherent conditional probability assessments}
\label{sect:2.1}
In real world applications, we very often have to manage uncertainty about the facts, which are described by (non-ambiguous) logical propositions.  For dealing with unknown facts we use  the notion of event. In formal terms, an event $A$  is a two-valued logical entity which can be \emph{true}, or \emph{false}.
The indicator of $A$, denoted by the same symbol, is  1, or 0, according to  whether $A$ is true, or false. The sure event and the impossible event are denoted by $\Omega$ and  $\emptyset$, respectively. 
Given two events $A$ and $B$,  we denote by $A\land B$, or simply by $AB$, (resp., $A \vee B$) the logical conjunction (resp., the  logical disjunction).   The  negation of $A$ is denoted by $\no{A}$.  We simply write $A \subseteq B$ to denote that $A$ logically implies $B$, that is  $A\no{B}=\emptyset$. 
We recall that  $n$ events $A_1,\ldots,A_n$ are logically independent when the number $m$ of constituents, or possible worlds, generated by them  is $2^n$ (in general  $m\leq 2^n$).

Given two events $A,H$,
with $H \neq \emptyset$, the conditional event $A|H$
is defined as a three-valued logical entity which is \emph{true}, or
\emph{false}, or \emph{void}, according to whether $AH$ is true, or $\no{A}H$
is true, or $\no{H}$ is true, respectively.
The notion of logical inclusion among events has been generalized to conditional events by Goodman and Nguyen in \cite{GoNg88} (see also \cite{gilio13ins} for some related results). 
Given two conditional events 
$A|H$ and $B|K$, we say that $A|H$ implies $B|K$, denoted by $A|H \subseteq B|K$, iff $AH$ {\em true} implies $BK$ {\em true} and $\no{B}K$ {\em true} implies $\no{A}H$ {\em true}; i.e., iff $AH \subseteq BK$ and $\no{B}K \subseteq \no{A}H$.
In the subjective approach to probability based on the betting scheme, to assess  $P(A|H)=x$ means that, for every real number $s$,  you are willing to pay (resp., to receive) 
an amount $sx$ and to receive (resp., to pay) $s$, or 0, or $sx$, according
to whether $AH$ is true, or $\widebar{A}H$ is true, or $\widebar{H}$
is true (bet called off), respectively. The random gain, which is  the difference between the (random) amount that you receive and the amount that you pay, is 
\[
G=(sAH+0\widebar{A}H+sx\widebar{H})-sx=
sAH+sx(1-H)-sx=
sH(A-x).
\]
In what follows, we assume that the probabilistic assessments are coherent (see Definition \ref{COER-RQ}). In particular, the coherence of any assessment $x=P(A|H)$ is equivalent to $min \; \G_{H} \; \leq 0 \leq max \;
\G_{H}$, $\forall \, s$, where $\G_{H}$ is the set of values of $G$ restricted to $H$. Then, the set $\Pi$ of coherent assessments $x$ on $A|H$ is: $(i)$ $\Pi=[0,1]$, when $\emptyset \neq AH \neq H$;
$(ii)$ $\Pi=\{0\}$, when $ AH=\emptyset$; $(iii)$ $\Pi=\{1\}$, when $ AH=H$. 
 In numerical terms, once  $x=P(A|H)$ is assessed by the betting scheme, the indicator of $A|H$, denoted by the same symbol, is defined as  $1$, or $0$, or $x$, 
according
to whether $AH$ is true, or $\widebar{A}H$ is true, or $\widebar{H}$
is true. 
Then, by setting  $P(A|H)=x$, 
\begin{equation}\label{EQ:AgH}
A|H=AH+x \no{H}=\left\{\begin{array}{ll}
1, &\mbox{if $AH$ is true,}\\
0, &\mbox{if $\no{A}H$ is true,}\\
x, &\mbox{if $\no{H}$ is true,}\\
\end{array}
\right.
\end{equation} and when you pay $sx$ the amount that you receive is $s\,A|H=s(AH+x\no{H})$, with a random gain given by  $G=sH(A-x)=s(A|H-x)$.  In particular, when $s=1$, you pay $x$ and receive $A|H$.
Notice that,  when $H\subseteq A$ (i.e., $AH=H$), by coherence $P(A|H)=1$ and hence $A|H=H+\no{H}=1$.  The negation of a conditional event $A|H$ is defined by  $\no{A}|H$, which coincides with  $1-A|H$.  
\begin{remark}\label{REM:AgH}
We point out that  the definition of (the indicator of) $A|H$ is not circular because, by the betting scheme, 
the three-valued numerical entity  $A|H$ is defined once the value $x=P(A|H)$ is assessed. 
Moreover, denoting prevision by the symbol $\prev$, the value $x$ coincides with the conditional prevision $\prev(A|H)$ because
\begin{equation*}\label{EQ:prevAgH}
\prev(A|H)=\prev(AH+x\no{H})=P(AH)+xP(\no{H})=P(A|H)P(H)+xP(\no{H})=xP(H)+xP(\no{H})=x.
\end{equation*}
We recall that 
a systematic study of the third value
  of a conditional event has been developed in \cite{CoSc99a}, where   it has been shown that  this value  satisfies all the properties of a conditional probability. In addition,  extensions to conditional possibility  and to general conditional measures of uncertainty have been given in \cite{BMCM02,CoSc01}. The semantics of our approach, however, is probabilistic, and hence the third value for the indicator of $A|H$ is $P(A|H)$. Notice that, given two conditional events $A|H$ and $B|K$, it makes sense, for their indicators,  to check the inequality $A|H\leq B|K$. For instance, the inequality holds when the conditional events satisfy the Goodman and Nguyen relation, i.e. $A|H\subseteq B|K$. Indeed, in this case   coherence requires that $P(A|H)\leq P(B|K)$ and hence $A|H\leq B|K$ (see also \cite[Theorem 6]{gilio13ins}). 
\end{remark}
Given a probability function $P$ defined on an arbitrary family $\mathcal{K}$ of 
conditional events, consider a finite subfamily $\F = \{E_1|H_1, \ldots,E_n|H_n\} \subseteq \mathcal{K}$ and the vector
$\P=(p_1,\ldots, p_n)$, where $p_i = P(E_i|H_i)$ is the
assessed probability for the conditional event  $E_i|H_i$, $i\in \{1,\ldots,n\}$.
With the pair $(\F,\P)$ we associate the random gain $G =
\sum_{i=1}^ns_iH_i(E_i - p_i)=\sum_{i=1}^ns_i(E_i|H_i - p_i)$. We denote by $\G_{\mathcal{H}_n}$ the set of values of $G$ restricted to $\H_n= H_1 \vee \cdots \vee H_n$, i.e., the set of values of $G$ when $\H_n$ is true. Then, we
recall below the notion of coherence in the context of the  {\em  betting scheme}.
\begin{definition}\label{COER-EV}{\rm
		The function $P$ defined on $\mathcal{K}$ is coherent if and only if, $\forall n
		\geq 1$, $\forall \, s_1, \ldots,
		s_n$, $\forall \, \F=\{E_1|H_1, \ldots,E_n|H_n\} \subseteq \mathcal{K}$, it holds that: $min \; \G_{\mathcal{H}_n} \; \leq 0 \leq max \;
		\G_{\mathcal{H}_n}$. }
\end{definition}
As shown by Definition \ref{COER-EV}, the function $P$ is coherent if and only if, in any finite combination of $n$ bets, 
it cannot happen that the values in the set $\G_{\mathcal{H}_{n}}$  are all
positive, or all negative (\emph{no Dutch Book}).  
In other words, in any finite combination of $n$ bets, 
after discarding the case where  all the bets are called off, the values of the random gain are  neither all positive nor all negative.
 
\subsection{Conditional random quantities and coherent conditional prevision assessments}
More in general, if $A$ is replaced by  a  random quantity $X$, by recalling that $\prev$ is the symbol of prevision,  to assess $\prev(X|H)=\mu$ means that, for every real number $s$,  you are willing to pay 
an amount $s\mu$ and to receive  $sX$, or $s\mu$, according
to whether $H$ is true, or $\widebar{H}$
is true (the bet is called off), respectively. Of course, when $X$ is an event $A$, it holds that $\prev(X|H)=P(A|H)$.
The random gain is $G=s(XH+\mu \widebar{H})-s\mu=
sH(X-\mu)$. 
By following the approach given in \cite{CoSc99a,GOPS16,GiSa13c,GiSa13a,GiSa14,lad96}, once  a coherent assessment $\mu=\prev(X|H)$ is specified,  the conditional random quantity $X|H$ (is not looked at as the restriction to $H$, but)
is defined as  $X$, or $\mu$,
according
to whether $H$ is true, or $\widebar{H}$
is true; that is,  
\begin{equation}\label{EQ:XgH}
X|H=XH+\mu \widebar{H}.
\end{equation}
As shown in (\ref{EQ:XgH}), given any random quantity $X$ and any event $H\neq \emptyset$, in the framework of subjective probability, in order to define $X|H$ we just need to specify the value $\mu$ of the conditional prevision $\prev(X|H)$.  Indeed, once  the value $\mu$ is specified, the object $X|H$ is (subjectively) determined.
We observe that (\ref{EQ:XgH}) is consistent because
\begin{equation}\label{EQ:PREVXgH}
\prev(XH+\mu\no{H})=\prev(XH)+\mu P(\no{H})=\prev(X|H)P(H)+\mu P(\no{H})=\mu P(H)+\mu P(\no{H})=\mu.
\end{equation}
By (\ref{EQ:XgH}), the random gain associated with a bet on $X|H$ can be represented as $G=s(X|H-\mu)$, that is $G$ is the difference between what you 
receive, $s X|H$, and what you pay, $s \mu$.
In what follows,
for any given conditional random quantity $X|H$, we assume that, when $H$ is true, the set of possible values of $X$ is finite.
 In this case we say 
that $X|H$ is a finite conditional random quantity.  
Denoting by $\mathcal{X}_H=\{x_{1}, \ldots,x_{r}\}$ the set of possible values  of $X$ restricted to $H$ and by setting $A_{j} = (X = x_{j})$, $j=1,\ldots, r$, it holds that 
$\bigvee_{j=1}^r A_j=H$ and
\begin{equation}
X|H=XH+\mu \widebar{H}=x_1A_1+\cdots +x_rA_r+\mu\widebar{H}.
\end{equation}
Given a prevision function $\prev$ defined on an arbitrary family $\mathcal{K}$ of finite
conditional random quantities, consider a finite subfamily $\F = \{X_1|H_1, \ldots,X_n|H_n\} \subseteq \mathcal{K}$ and the vector
$\M=(\mu_1,\ldots, \mu_n)$, where $\mu_i = \prev(X_i|H_i)$ is the
assessed prevision for the conditional random quantity $X_i|H_i$, $i\in \{1,\ldots,n\}$.
With the pair $(\F,\M)$ we associate the random gain $G =
\sum_{i=1}^ns_iH_i(X_i - \mu_i)=\sum_{i=1}^ns_i(X_i|H_i - \mu_i)$. We  denote by $\G_{\mathcal{H}_n}$ the set of values of $G$ restricted to $\H_n= H_1 \vee \cdots \vee H_n$. 
Then,  the notion of coherence is defined as below.
\begin{definition}\label{COER-RQ}{\rm
		The function $\prev$ defined on $\mathcal{K}$ is coherent if and only if, $\forall n
		\geq 1$, $\forall \, s_1, \ldots,
		s_n$, $\forall \, \F=\{X_1|H_1, \ldots,X_n|H_n\} \subseteq \mathcal{K}$,   it holds that: $min \; \G_{\mathcal{H}_n} \; \leq 0 \leq max \;
		\G_{\mathcal{H}_n}$. }
\end{definition}
In particular, by Definition \ref{COER-RQ}, the coherence of a prevision assessment $\prev(X|H)=\mu$ is equivalent to $ \min \mathcal{X}_{H} \leq \mu \leq \max \mathcal{X}_{H}$, where we recall that $\mathcal{X}_{H}$ is the set of valued of $X$ when $H$ is true.

Given a family $\F = \{X_1|H_1,\ldots,X_n|H_n\}$, for each $i \in \{1,\ldots,n\}$ we denote by $\{x_{i1}, \ldots,x_{ir_i}\}$ the set of possible values  of $X_i$ when  $H_i$ is true; then,  we set $A_{ij} = (X_i = x_{ij})$,   $i=1,\ldots,n$, $j = 1, \ldots, r_i$. We set $C_0 = \widebar{H}_1 \cdots \widebar{H}_n$ (it may be $C_0 = \emptyset$) and  we denote by $C_1, \ldots, C_m$ the constituents
contained in $\H_n=H_1\vee \cdots \vee H_n$. Hence
$\bigwedge_{i=1}^n(A_{i1} \vee \cdots \vee A_{ir_i} \vee
\widebar{H}_i) = \bigvee_{h = 0}^m C_h$.
With each $C_h,\, h \in \{1,\ldots,m\}$, we associate a vector
$Q_h=(q_{h1},\ldots,q_{hn})$, where $q_{hi}=x_{ij}$ if $C_h \subseteq
A_{ij},\, j=1,\ldots,r_i$, while $q_{hi}=\mu_i$ if $C_h \subseteq \widebar{H}_i$;
with $C_0$ we associate  $Q_0=\M = (\mu_1,\ldots,\mu_n)$.
Denoting by $\I$ the convex hull of $Q_1, \ldots, Q_m$, the condition  $\M\in \I$ amounts to the existence of a vector $(\lambda_1,\ldots,\lambda_m)$ such that:
$ \sum_{h=1}^m \lambda_h Q_h = \M \,,\; \sum_{h=1}^m \lambda_h
= 1 \,,\; \lambda_h \geq 0 \,,\; \forall \, h$; in other words, $\M\in \I$ is equivalent to the solvability of the system $(\Sigma)$, associated with  $(\F,\M)$,
\begin{equation}\label{SYST-SIGMA}
(\Sigma) \quad
\begin{array}{ll}
\sum_{h=1}^m \lambda_h q_{hi} =
\mu_i \,,\; i \in\{1,\ldots,n\} \,, 
\sum_{h=1}^m \lambda_h = 1,\;\;\lambda_h \geq 0 \,,\;  \,h \in\{1,\ldots,m\}\,.
\end{array}
\end{equation}
Given the assessment $\M =(\mu_1,\ldots,\mu_n)$ on  $\F =
\{X_1|H_1,\ldots,X_n|H_n\}$, let $S$ be the set of solutions $\Lambda = (\lambda_1, \ldots,\lambda_m)$ of system $(\Sigma)$.   
We point out that the solvability of  system $(\Sigma)$  is a necessary (but not sufficient) condition for coherence of $\M$ on $\F$. When $(\Sigma)$ is solvable, that is  $S \neq \emptyset$, we define:
\begin{equation}\label{EQ:I0}
\begin{array}{ll}
I_0 = \{i : \max_{\Lambda \in S}  \sum_{h:C_h\subseteq H_i}\lambda_h= 0\},\;
\F_0 = \{X_i|H_i \,, i \in I_0\},\;\;  \M_0 = (\mu_i ,\, i \in I_0)\,.
\end{array}
\end{equation}
For what concerns the probabilistic meaning of $I_0$, it holds that  $i\in I_0$ if and only if the (unique) coherent extension of $\M$ to $H_i|\H_n$ is zero.
Then, the following theorem can be proved  (\cite[Theorem 3]{BiGS08}):
\begin{theorem}[\emph{Operative characterization of coherence}]
	\label{CNES-PREV-I_0-INT}{\rm 
		A conditional prevision assessment ${\M} = (\mu_1,\ldots,\mu_n)$ on
		the family $\F = \{X_1|H_1,\ldots,X_n|H_n\}$ is coherent if
		and only if the following conditions are satisfied: \\
		(i) the system $(\Sigma)$ defined in (\ref{SYST-SIGMA}) is solvable; (ii) if $I_0 \neq \emptyset$, then $\M_0$ is coherent. }
\end{theorem}
Coherence can be characterized in terms of proper scoring rules (\cite{BiGS12,GiSa11a}), which can be related  to the notion of entropy in information theory (\cite{LS18,LaSA12,LSA15,LaSA18}).

 The  result below (\cite[Theorem 4]{GiSa14}) shows that
if  
 two conditional random quantities $X|H$, $Y|K$ coincide when  $H\vee K$ is true, then $X|H$ and $ Y|K$ also coincide when  $H\vee K$ is false, and hence $X|H$ coincides with $Y|K$ in all cases.
\begin{theorem}\label{THM:EQ-CRQ}{\rm Given any events $H\neq \emptyset$ and  $K\neq \emptyset$, and any random quantities $X$ and $Y$, let $\Pi$ be the set of the coherent prevision assessments $\prev(X|H)=\mu$ and $\prev(Y|K)=\nu$. \\
		$(i)$ Assume that, for every $(\mu,\nu)\in \Pi$,  $X|H=Y|K$  when  $H\vee K$ is true; then   $\mu=\nu$ for every $(\mu,\nu)\in \Pi$. \\
		$(ii)$ For every $(\mu,\nu)\
		\in \Pi$,   $X|H=Y|K$  when  $H\vee K$ is true  if and only if $X|H=Y|K$.
}\end{theorem}	
\begin{remark}\label{REM:INEQ-CRQ}
Theorem \ref{THM:EQ-CRQ} has been generalized in \cite[Theorem 6]{GiSa19} by replacing the symbol ``$=$'' by ``$\leq$'' in statements  $(i)$ and $(ii)$. In other words,
if  $X|H\leq Y|K$ when  $H\vee K$ is true, then $\prev(X|H)\leq  \prev(Y|K)$ and hence  $X|H\leq  Y|K$ in all cases.	
\end{remark}
\subsection{Conjoined  and  iterated conditionals}
We recall now  the notions of conjoined (e.g.,  \emph{$($if $H$ then $A)$ and $($if $K$ then $B)$})  and iterated conditionals (e.g.,  \emph{if $($if $H$ then $A)$, then $($if $K$ then $B)$}), which were introduced  in the framework of conditional  random quantities (\cite{GiSa13c,GiSa13a,GiSa14}). 
Given a coherent probability assessment $(x,y)$ on $\{A|H,B|K\}$ we consider the random quantity $AHBK+x\no{H}BK+y\no{K}AH$ and we set $\prev[(AHBK+x\no{H}BK+y\no{K}AH)|(H\vee K)]=z$. Then  we define the  conjunction $(A|H)\wedge(B|K)$ as follows:
\begin{definition}\label{CONJUNCTION}{\rm Given a coherent prevision assessment 
$P(A|H)=x$, $P(B|K)=y$, and $\prev[(AHBK+x\no{H}BK+y\no{K}AH)|(H\vee K)]=z$, the conjunction
 $(A|H)\wedge(B|K)$ is the conditional random quantity  defined as
\begin{equation}\label{EQ:CONJUNCTION}
\begin{array}{ll}
(A|H)\wedge(B|K)=(AHBK+x\no{H}BK+y\no{K}AH)|(H\vee K) =\\
=(AHBK+x\no{H}BK+y\no{K}AH)(H\vee K)+z\no{H}\,\no{K}=
\left\{\begin{array}{ll}
1, &\mbox{if $AHBK$ is true,}\\
0, &\mbox{if $\no{A}H\vee \no{B}K$ is true,}\\
x, &\mbox{if $\no{H}BK$ is true,}\\
y, &\mbox{if $AH\no{K}$ is true,}\\
z, &\mbox{if $\no{H}\no{K}$ is true}.
\end{array}
\right.
\end{array}
\end{equation}
}\end{definition}
Notice that in  (\ref{EQ:CONJUNCTION})  the conjunction is represented as $X|H$ is in (\ref{EQ:XgH}) and, once  the (coherent) assessment $(x,y,z)$ is given, the  
conjunction $(A|H)\wedge (B|K)$
is (subjectively) determined. Moreover,  by (\ref{EQ:PREVXgH}),
it holds that $\prev[(A|H)\wedge(B|K)]=z$.  
 We recall that, in betting terms,  $z$ represents the amount you agree to pay, with the proviso that you will receive the quantity
\begin{equation}
(A|H)\wedge(B|K)=AHBK+x\no{H}BK+y\no{K}AH+z\no{H}\no{K},
\end{equation}
 which assumes one of the following values:
\begin{itemize}
	\item $1$, if both conditional events are true;
	\item  $0$, if at least one of the conditional events is false; 
	\item  the probability of  the conditional event that is void, if one conditional event is void  and the other one is true;
	\item  $z$ (the amount that you payed), if both 
	conditional events are void.
\end{itemize}
The  result  below  shows that Fr\'echet-Hoeffding bounds still hold for the conjunction of two conditional events (\cite[Theorem~7]{GiSa14}).
\begin{theorem}\label{THM:FRECHET}{\rm
		 Given any coherent assessment $(x,y)$ on $\{A|H, B|K\}$, with $A,H,B,K$ logically independent, and with $H \neq \bot, K \neq \bot$, the extension $z = \mathbb{P}[(A|H) \wedge (B|K)]$ is coherent if and only if   the following  Fr\'echet-Hoeffding bounds are satisfied:
	\begin{equation}\label{LOW-UPPER}
\max\{x+y-1,0\} =\;z' \leq \; z \; \leq \; z''=\; \min\{x,y\} \,.
	\end{equation}
}\end{theorem}
We observe that, by logical independence, the assessment $(x,y)$ on $\{A|H,B|K\}$ is coherent for every $(x,y)\in[0,1]^2$. Moreover, the main aspect in the  proof of Theorem \ref{THM:FRECHET} shows that  the assessment $(x,y,z)$ is coherent if and only if it belongs to the  tetrahedron with vertices the points $(1, 1, 1)$, $(1, 0, 0)$, $(0, 1, 0)$, $(0, 0, 0)$.
\begin{remark}
Notice that, the assumption of logical independence plays a key role for the validity of Theorem    \ref{THM:FRECHET}. Indeed, in case of some logical dependencies, for the interval 	 $[z',z'']$ of coherent extensions $z$, it holds that $\max\{x+y-1,0\}\leq z'\leq z''\leq\min\{x,y\}$.
For instance, when $H=K$ and $AB=\emptyset$, the coherence of the assessment $(x,y)$ on $\{A|H,B|H\}$ is equivalent to the condition  $x+y-1\leq 0$. In this case, it holds that $(A|H)\wedge (B|H)=AB|H$ with $P(AB|H)=0$;  then, 
 the unique coherent extension on $AB|H$ is $z=0$.
As another example, in the case $A=B$, with $A,H,K$ assumed to be  logically independent, it holds that the assessment $(x,y)$ on $\{A|H,A|K\}$ is coherent for every $(x,y)\in[0,1]^2$. Moreover,
the extension $z$ is coherent if and only if $xy\leq  z\leq \min\{x,y\}$ (see \cite[Theorem 5]{GiSa19b}). Finally, we remark that in all cases,  for each coherent extension $z$, it holds that  $z\in [z',z'']\subseteq  [0,1]$; thus $(A|H) \wedge (B|K)\in [0,1]$.
\end{remark} 
Other approaches to compounded conditionals, which are not based on coherence, can be found in \cite{kaufmann09,mcgee89} (see also \cite{Cala17,Godo17}).
A study of the lower and upper bounds for other definitions of conjunction, where the conjunction is a conditional event like Adams' quasi conjunction,
 has been given in \cite{SUM2018S}.

 The notion of an iterated conditional is based on a  structure  like  (\ref{EQ:AgH}), 
i.e. $ \Box|\bigcirc = \Box\wedge \bigcirc +\prev(\Box|\bigcirc)\no{\bigcirc}$,  where $\Box$ denotes $B|K$ and  $\bigcirc$ denotes $A|H$, and where we set $\prev(\Box|\bigcirc)=\mu$.  We recall that in the framework of subjective probability $\mu=\prev(\Box|\bigcirc)$ is the amount that you agree to pay, by knowing that you will receive the random quantity  $\Box\wedge \bigcirc +\prev(\Box|\bigcirc)\no{\bigcirc}$. 
Then, the  notion of iterated conditional $(B|K)|(A|H)$ is defined (see, e.g., \cite{GiSa13c,GiSa13a,GiSa14}) as follows:
\begin{definition}
	\label{DEF:ITER-COND}
	 Given any  pair of conditional events $A|H$ and $B|K$, with $AH\neq \emptyset$, let $(x,y,z)$ be a coherent assessment on $\{A|H,B|K, (A|H)\wedge (B|K)\}$. The iterated
	conditional $(B|K)|(A|H)$ is defined as 
	\begin{equation}\label{EQ:ITER-COND}
	(B|K)|(A|H) = (B|K) \wedge (A|H) + \mu\,  \no{A}|H=\left\{\begin{array}{ll}
	1, &\mbox{if $AHBK$ is true,}\\
	0, &\mbox{if $AH\no{B}K$ is true,}\\
	y, &\mbox{if $AH\no{K}$ is true,}\\
	x + \mu(1-x), &\mbox{if $\no{H}BK$ is true,}\\
	\mu(1-x), &\mbox{if $\no{H}\no{B}K$ is true,}\\
	z+\mu(1-x),  &\mbox{if $ \no{H}\no{K}$ is true,}\\
	\mu,  &\mbox{if $\no{A}H$ is true,}	
	\end{array}
	\right.
	\end{equation}
	where
	$\mu =\mathbb{P}[(B|K)|(A|H)]=\prev[(B|K) \wedge (A|H) + \mu\,  \no{A}|H]$.
\end{definition}
In Definition \ref{DEF:ITER-COND} the quantity $\mu$, unlike $z$ in Definition \ref{CONJUNCTION}, is not introduced as a (suitable) conditional prevision by a preliminary step. However, later we examine some aspects of the  coherence of $\mu$, and  in Theorem \ref{THM:PIONITERATED} we illustrate  the coherence of the prevision assessment $(x,y,z,\mu)$, under the logical independence of $A,H,B,K$.  We recall that,  by the  betting scheme, if you  assess $\prev[(B|K)|(A|H)]=\mu$, then you agree  to pay the amount $\mu$, with the proviso that you will receive the quantity $(A|H)\wedge(B|K)  + \mu\, \no{A}|H$. 
In what follows,  in order to check coherence,   a bet is  called off when you  receive  back the  paid  amount $\mu$, for every  $\mu$. 
 This may  happen in particular cases and for some preliminary probability values.
For instance, when $H=K=\Omega$, by Definition \ref{DEF:ITER-COND}  it holds that 
$(B|\Omega)|(A|\Omega)=AB+\mu\no{A}$, where $\mu=\prev(AB+\mu\no{A})$.  Then,
in a bet on $AB+\mu\no{A}$ you agree to pay $\mu$  by receiving  1, or 0, or $\mu$ according to whether $AB$ is true, or $A\no{B}$ is true, or  $\no{A}$ is true. 
When  $\no{A}$ is true you receive back the same amount you paid and hence, in this case, the bet is called off.
Then, for checking the coherence of $\mu$, we discard the case where $\no{A}$ is true, by only considering the values of the random gain restricted to $A$.  Thus, $\mu=P(B|A)$ and  $(B|\Omega)|(A|\Omega)=AB+\mu\no{A}=AB+P(B|A)\no{A}=B|A$.
Moreover, as we will see, a  bet on  $(B|K)|(A|H)$  may  be also called off     when $x=P(A|H)=0$.

Definition~\ref{DEF:ITER-COND}  allows us to represent antecedent-nested and consequent-nested conditionals. These are, respectively, conditionals with other conditionals as antecedents, and conditionals with other conditionals as consequents. As an  example of a natural language instantiation of such a conditional consider the following:
\begin{quote}``If the match is canceled if it starts raining, then the match is canceled if it starts snowing'' (p. 45 of \cite{douven16}).\end{quote} 
\begin{remark}\label{REM:A|H=0}
	Notice that we assume  $AH\neq \emptyset$ to give a nontrivial meaning to the notion of the iterated conditional. Indeed,   if  $AH$ were equal to $\emptyset$ (and of course $H\neq \emptyset$),  that is $A|H=0$, then it would be the case that $\no{A}|H =1$ and     
	$(B|K)|(A|H)=(B|K)|0=(B|K) \wedge (A|H) + \mu \no{A}|H=\mu$ would follow; that is, $(B|K)|(A|H)$ 
	would coincide with the (indeterminate) value $\mu$.  Similarly in the case of  $B|\emptyset$   (which is of no interest). Thus  the trivial iterated conditional $(B|K)|0$
	is not considered in our approach. In betting terms, both situations mean that you get your money back because these bets are always called-off.
\end{remark}
We observe that, by the linearity of prevision, it holds that 
\[
\begin{array}{lll}
\mu &=& \prev((B|K)|(A|H)) =\prev((B|K) \wedge (A|H)+\mu\, \no{A}|H)= 
\prev((B|K) \wedge (A|H))+\prev(\mu\, \no{A}|H)
=\\&=&\prev((B|K) \wedge (A|H)) + \mu\, P(\no{A}|H) = z + \mu(1-x) \,,
\end{array}
\]
from which it follows that  (\cite{GiSa13a})
\begin{equation}\label{EQ:PRODUCT}
z=\prev((B|K) \wedge (A|H))=\mu x= \prev((B|K)|(A|H))P(A|H).
\end{equation}
Here, when $x>0$, we obtain $\mu = \frac{z}{x} \in [0,1]$.
Based on the equality $\mu=z + \mu(1-x)$,
formula (\ref{EQ:ITER-COND}) can be written as
\begin{equation}\label{EQ:SHORTITER-COND}
(B|K)|(A|H) 
= \left\{\begin{array}{ll}
1, &\mbox{if $AHBK$ is true,}\\
0, &\mbox{if $AH\no{B}K$ is true,}\\
y, &\mbox{if $AH\no{K}$ is true,}\\
x + \mu(1-x), &\mbox{if $\no{H}BK$ is true,}\\
\mu(1-x), &\mbox{if $\no{H}\no{B}K$ is true,}\\
\mu,  &\mbox{if $\no{A}H \vee \no{H}\no{K}$ is true.}
\end{array}
\right.
\end{equation}
When $x>0$ it holds that  $\{1,0,y,x+\mu(1-x),\mu(1-x),\mu\}\subset [0,1]$ and hence $(B|K)|(A|H) \in[0,1]$.
Moreover, when $x=0$, it holds that
\begin{equation}
(B|K)|(A|H) 
= \left\{\begin{array}{ll}
1, &\mbox{if $AHBK$ is true,}\\
0, &\mbox{if $AH\no{B}K$ is true,}\\
y, &\mbox{if $AH\no{K}$ is true,}\\
\mu, &\mbox{if $\no{A}H \vee \no{H}$ is true.}
\end{array}
\right.
\end{equation}
In order that the prevision assessment  $\mu$ on $(B|K)|(A|H)$ be coherent, $\mu$ must belong to the convex hull of the values $0,y,1$; that is, (also when $x=0$) it must be that $\mu \in [0,1]$.  Then,  $(B|K)|(A|H)\in [0,1]$ in all cases.	
In general, in 
\cite[Theorem 3]{SPOG18} it has been given the following result:
	\begin{theorem}
	\label{THM:PIONITERATED}
	Let $A,B,H,K$ be any logically independent events. The set $\Pi $ of all
	coherent assessments $(x,y,z,\mu )$ on the family
	$\mathcal{F}=\{A|H,B|K$, $(A|H)\wedge (B|K), (B|K)|(A|H)\}$ is
	$\Pi =\Pi '\cup \Pi ''$, where
	\begin{equation}
	\label{EQ:PI}
	\begin{array}{l}
	\Pi '=\{(x,y,z,\mu ): x\in (0,1], y\in [0,1],
	z\in [z', z''], \mu =
	\frac{z}{x}\},
	\\ \noalign{\vspace{4pt}}
	\text{with } z'=\max\{x+y-1,0\}, z''= \min\{x,y\}, \text{ and}
	\\ \noalign{\vspace{4pt}}
	\Pi ''=\{(0,y,0,\mu ): (y,\mu )\in [0,1]^{2}\}.
	\end{array}
	\end{equation}
\end{theorem}

\begin{remark}
We note that  the iterated conditional $(B|K)|A$ is (not a conditional event but) a conditional random quantity. Moreover,   $(B|K)|A$ does not coincide with the conditional event $B|AK$ (see \cite[Section 3.3]{GiSa14}). 
Thus the import-export principle (\cite{mcgee89}), which says that   $(B|K)|A=B|AK$, does not hold (as, e.g., in  \cite{adams75,kaufmann09}). Therefore, as shown in \cite{GiSa14}, we  avoid the  counter-intuitive consequences related to the well-known Lewis' first triviality result  (\cite{lewis76}).
\end{remark}
\begin{remark}\label{REM:AHK}
As a further comment on the import-export principle, we observe that 
	given any random quantity  $X$ and any  events  $H,K$, with $H\neq \emptyset$, $K\neq \emptyset$, it holds that (see \cite[Proposition 1]{GiSa13c}): $(X|K)|H=(X|H)|K=X|HK$, 
	when $H\subseteq K$ or $K\subseteq H$. Of course, $X|HK$ coincides with $X|H$, or $X|K$, according to whether $H\subseteq K$, or 
	$K\subseteq H$, respectively.
	In particular, when $X$ is an event $A$, it holds that
	$(A|K)|H=(A|H)|K=A|HK$.
	Then, given any two events $H,K$, with $K\neq \emptyset$, as $K\subseteq H\vee K$,    it holds that: 
	$(X|K)|(H\vee K)=X|K$.
\end{remark}
We recall the notion of conjunction of $n$ conditional events (\cite{GiSa19}).
\begin{definition}\label{DEF:CONGn}	Let  $n$ conditional events $E_1|H_1,\ldots,E_n|H_n$ be given.
	For each  non-empty strict subset $S$  of $\{1,\ldots,n\}$,  let $x_{S}$ be a prevision assessment on $\bigwedge_{i\in S} (E_i|H_i)$.
	Then, the conjunction  $(E_1|H_1) \wedge \cdots \wedge (E_n|H_n)$ is the conditional random quantity $\C_{1\cdots n}$ defined as
	\begin{equation}\label{EQ:CF}
	\begin{array}{lll}
	\C_{1\cdots n}=
	[\bigwedge_{i=1}^n E_iH_i+\sum_{\emptyset \neq S\subset \{1,2\ldots,n\}}x_{S}(\bigwedge_{i\in S} \no{H}_i)\wedge(\bigwedge_{i\notin S} E_i{H}_i)]|(\bigvee_{i=1}^n H_i)=
	\\
	=\left\{
	\begin{array}{llll}
	1, &\mbox{ if } \bigwedge_{i=1}^n E_iH_i\, \mbox{ is true,} \\
	0, &\mbox{ if } \bigvee_{i=1}^n \no{E}_iH_i\, \mbox{ is true}, \\
	x_{S}, &\mbox{ if } (\bigwedge_{i\in S} \no{H}_i)\wedge(\bigwedge_{i\notin S} E_i{H}_i)\, \mbox{ is true}, \; \emptyset \neq S\subset \{1,2\ldots,n\},\\
	x_{1\cdots n}, &\mbox{ if } (\bigwedge_{i=1}^n \no{H}_i) \mbox{ is true},
	\end{array}
	\right.
	\end{array}
	\end{equation}	
\end{definition}
where 
\[
x_{1\cdots n}=x_{\{1,\ldots, n\}}=\prev[(\bigwedge_{i=1}^n E_iH_i+\sum_{\emptyset \neq S\subset \{1,2\ldots,n\}}x_{S}(\bigwedge_{i\in S} \no{H}_i)\wedge(\bigwedge_{i\notin S} E_i{H}_i))|(\bigvee_{i=1}^n H_i)]=\prev(\C_{1\cdots n}).
\]
Of course,  we obtain  $\C_1=E_1|H_1$, when $n=1$.  In  Definition \ref{DEF:CONGn}  each possible value $x_S$ of $\C_{1\cdots n}$,  $\emptyset\neq  S\subset \{1,\ldots,n\}$, is evaluated  when defining (in a previous step) the conjunction $\C_{S}=\bigwedge_{i\in S} (E_i|H_i)$. 
Then, after the conditional prevision $x_{1\cdots n}$ is evaluated, $\C_{1\cdots n}$ is completely specified.
In the framework of the betting scheme,  $x_{1\cdots n}$ is the amount  that you agree to pay  with the proviso that you will receive:
\begin{itemize}
	\item $1$, if all conditional events are true;
	\item	$0$, if at least one of the conditional events is false;
	\item the prevision of the conjunction of that conditional events which are void,  otherwise. In particular you receive back $x_{1\cdots n}$ when all  conditional events are void.
\end{itemize}

The operation of conjunction is associative and commutative. In addition, the following monotonicity property holds (\cite[Theorem 7]{GiSa19})
\begin{equation}\label{EQ:MONOTONY}
\C_{1\cdots n+1}\leq \C_{1\cdots n}.
\end{equation}
We recall the following generalized notion  of  iterated conditional (\cite[Definition 14]{GiSa19}). 
\begin{definition}\label{DEF:GENITER}
	Let  $n+1$ conditional events $E_1|H_1, \ldots, E_{n+1}|H_{n+1}$, with $(E_1|H_1) \wedge \cdots \wedge (E_n|H_n)\neq 0$,  be given. We denote by $(E_{n+1}|H_{n+1})|((E_1|H_1) \wedge \cdots \wedge (E_n|H_n))$   the following random quantity
	\[
	\begin{array}{ll}
	(E_1|H_1) \wedge \cdots \wedge (E_{n+1}|H_{n+1})  + \mu (1-(E_1|H_1) \wedge \cdots \wedge (E_n|H_n)).
	\end{array}
	\]
	where $\mu = \prev[(E_{n+1}|H_{n+1})|((E_1|H_1) \wedge \cdots \wedge (E_n|H_n))]$.
\end{definition}
Definition~\ref{DEF:GENITER} extends  the notion of the iterated conditional $(E_{2}|H_{2})|(E_1|H_1)$ given in 
Definition~\ref{DEF:ITER-COND} to the case where  the antecedent is a conjunction of conditional events.
Based on the betting metaphor,  the quantity $\mu$ is the amount to be paid in order to receive the amount $\C_{1\cdots n+1}+\mu (1-\C_{1\cdots n})=	(E_1|H_1) \wedge \cdots \wedge (E_{n+1}|H_{n+1})  + \mu (1-(E_1|H_1) \wedge \cdots \wedge (E_n|H_n))$. 
We  observe that, defining 	$\prev(\C_{1\cdots n})=z_{n}$ and 
$\prev(\C_{1\cdots n+1})=z_{n+1}$, by the linearity of prevision it holds that $\mu=z_{n+1}+\mu(1-z_{n})$; then, $z_{n+1}=\mu z_n$, that is 
\begin{equation}
\label{EQ:COMPOUNDCN}
\prev(\C_{1\cdots n+1})=\prev[(E_{n+1}|H_{n+1})|\C_{1\cdots n}]\prev(\C_{1\cdots n}),
\end{equation}
which is  the compound prevision theorem for the generalized iterated conditionals.

\section{The iterated conditional $D|(C|A)$}	
\label{SEC:DgCgA}
In this section we analyze the iterated conditional $D|(C|A)$ (see, e.g., \cite{douven10,kaufmann09}) which is a more general version of the iterated conditional of interest $A|(C|A)$ that will be studied in Section \ref{SEC:AgCgA}.
We examine the object $D|(C|A)$ in the framework of the betting scheme.
Given   any real number $x\in[0,1]$,  we denote by $(x>0)$ the event  which is true or false, according to whether $x$ is positive or zero, respectively. By the symbol $(x>0)E$ we denote the conjunction between  $(x>0)$ and any event  $E$. 
Then, the event $AC\vee (x>0)\no{A}$ coincides with $AC\vee \no{A}$, or $AC$, according to whether $x$ is positive or zero, respectively. 
The next result  shows that    $D|(C|A)$ is a conditional random quantity, where the (dynamic) conditioning event is
$AC\vee (x>0)\no{A}$.
\begin{theorem}\label{THM:DgCgADynamic}
Let  a coherent assessment  $(x,\mu)$ on $\{C|A,D|(C|A)\}$ be given.
The iterated conditional $D|(C|A)$ is  the  conditional random quantity 
\begin{equation}\label{EQ:DgCgADynamic}
\begin{array}{lll}
D|(C|A)&=&
[ACD+(x+\mu(1-x))\no{A}D+\mu(1-x)\no{A}\no{D}]|[AC\vee (x>0)\no{A}]=\\ \\
&=&\left\{
\begin{array}{ll}
D|AC, & \mbox{ if } x=0,\\ 
\,[ACD+(x+\mu(1-x))\no{A}D+\mu(1-x)\no{A}\no{D}]|(AC\vee \no{A}), & \mbox{ if } 0<x<1,\\
D|(AC\vee\no{A}),& \mbox{ if } x=1.\\
\end{array}
\right.
\end{array}
\end{equation}	
\end{theorem}
\begin{proof}
Let us  consider a bet on $D|(C|A)$, with $A\neq\bot, C\neq \bot$,  $P(C|A)=x$, and $\prev[D|(C|A)]=\mu$. In this bet, $\mu$ is the amount that you agree to pay, while  $D|(C|A)$ is the amount that you  receive.
We observe that the bet on $D|(C|A)$ must be called off in all cases where the random gain  $G=D|(C|A)-\prev[D|(C|A)]=D|(C|A)-\mu$ coincides with zero, whatever be the assessed value $\mu$. In other words,  the  bet  must be called off in all cases where the amount 
 $D|(C|A)$ that you receive coincides  with the quantity that you payed  $\mu$, whatever be the assessed value $\mu$. 
By applying
(\ref{EQ:SHORTITER-COND})  to the iterated conditional $(D|\Omega)|(C|A)=D|(C|A)$, we obtain 
\begin{equation}\label{EQ:DgCgA}
D|(C|A)= D\wedge (C|A)+\mu(1-C|A)=\left\{\begin{array}{ll}
1, &\mbox{ if }  ACD \mbox{ is true,}\\
0, &\mbox{ if }  AC\no{D} \mbox{ is true,}\\
\mu, &\mbox{ if }  A\no{C}\mbox{ is true,}\\
x+\mu(1-x), &\mbox{ if }  \no{A}D \mbox{ is true,}\\
\mu(1-x), &\mbox{ if }  \no{A}\no{D} \mbox{ is true.}
\end{array}
\right.
\end{equation} 
 We distinguish three cases: $(i)$ $x=0$; $(ii)$ $0<x<1$; $(iii)$ $x=1$. 
We show that in case $(i)$ the bet is called off when $\no{A}\vee \no{C}$ is true; in cases $(ii)$ and $(iii)$ the bet is called off when $\no{A}C$ is true.\\
Case $(i)$. 
As $x=0$, formula (\ref{EQ:DgCgA}) becomes
\begin{equation}\label{EQ:CgBgHx=0}
D|(C|A)=\left\{\begin{array}{ll}
1, &\mbox{ if }  ACD \mbox{ is true,}\\
0, &\mbox{ if }  AC\no{D} \mbox{ is true,}\\
\mu, &\mbox{ if }  A\no{C}\mbox{ is true,}\\
\mu, &\mbox{ if }  \no{A}D \mbox{ is true,}\\
\mu, &\mbox{ if }  \no{A}\no{D} \mbox{ is true,}
\end{array}
\right.
=\left\{\begin{array}{ll}
1, &\mbox{ if }  ACD \mbox{ is true,}\\
0, &\mbox{ if }  AC\no{D} \mbox{ is true,}\\
\mu, &\mbox{ if }  \no{A}\vee \no{C} \mbox{ is true}.\\
\end{array}
\right.
\end{equation}
By setting $t=P(D|AC)$ it holds that	
\begin{equation}\label{EQ:CgBH}
\small
D|AC=\left\{\begin{array}{ll}
1, &\mbox{ if }  ACD \mbox{ is true,}\\
0, &\mbox{ if }  AC\no{D} \mbox{ is true,}\\
t, &\mbox{ if }  \no{A}\vee \no{C} \mbox{ is true}.\\
\end{array}
\right.
\end{equation}	
Then, based on  Theorem \ref{THM:EQ-CRQ}, from   (\ref{EQ:CgBgHx=0}) and (\ref{EQ:CgBH})   it follows that  $\mu=t$ and hence
\begin{equation}\label{EQ:DgCgAifx=0}
D|(C|A)=D|AC, \;\mbox{ if } x=0. 
\end{equation}
That is, the two objects $D|(C|A)$ and  $D|AC$ coincide when $x=0$. Here, the bet on $D|(C|A)$ is called off when  $\no{AC}=\no{A}\vee \no{C}$  is true.
\\
Case $(ii)$. 
As $0<x<1$, from (\ref{EQ:DgCgA})   the equality $D|(C|A)=\mu$ holds  for every $\mu$ only when $A\no{C}$ is true. This means that   the bet on $D|(C|A)$ is called off when $A\no{C}$ is true.
Then  the conditioning event is  $\no{A\no{C}}=C\vee \no{A}=AC\vee \no{A}$ and hence
\begin{equation}\label{EQ:DgCgAif0<x<1}
D|(C|A)=
[ACD+(x+\mu(1-x))\no{A}D+\mu(1-x)\no{A}\no{D}]|[AC\vee \no{A}], \mbox{ if } 0<x<1.
\end{equation}
Case $(iii)$. 
As  $x=1$,  the iterated conditional $D|(C|A)$ coincides with the conditional event  $D|(AC\vee \no{A})$. Indeed, in this case formula (\ref{EQ:DgCgA}) becomes
\begin{equation}\label{EQ:CgBgHx=1}
\small
D|(C|A)= \left\{\begin{array}{ll}
1, &\mbox{ if }  ACD \mbox{ is true,}\\
0, &\mbox{ if }  AC\no{D} \mbox{ is true,}\\
\mu, &\mbox{ if }  A\no{C}\mbox{ is true,}\\
1, &\mbox{ if }  \no{A}D \mbox{ is true,}\\
0, &\mbox{ if }  \no{A}\no{D} \mbox{ is true,}
\end{array}
\right.=\left\{\begin{array}{ll}
1, &\mbox{ if }  (AC\vee\no{A}) D \mbox{ is true,}\\
0, &\mbox{ if }  (AC\vee\no{A}) \no{D}  \mbox{ is true,}\\
\mu, &\mbox{ if }  A\no{C}\mbox{ is true.}\\
\end{array}
\right.
\end{equation}	
In addition, by setting $\nu=P(D|(AC\vee \no{A}))$ it holds that, 
\begin{equation}\label{EQ:MI}
\small
D|(AC\vee \no{A})=\left\{\begin{array}{ll}
1, &\mbox{ if }  (AC\vee \no{A}) D \mbox{ is true,}\\
0, &\mbox{ if }  (AC\vee \no{A}) \no{D}  \mbox{ is true,}\\
\nu, &\mbox{ if }  A\no{C}\mbox{ is true.}\\
\end{array}
\right.
\end{equation}	
Then, based on  Theorem \ref{THM:EQ-CRQ}, from   (\ref{EQ:CgBgHx=1}) and (\ref{EQ:MI})   it follows that  $\mu=\nu$ and hence
\begin{equation}\label{EQ:DgCgAifx=1}
D|(C|A)=D|(AC\vee \no{A}), \mbox{ if } x=1. 
\end{equation}
We remark that  (\ref{EQ:DgCgAifx=1}) also  follows by observing that  when $x=1$ it holds that $C|A=AC+x\no{A}=AC+\no{A}=AC\vee A$. Then, we directly obtain $D|(C|A)=D|(AC\vee \no{A})$.\\
Finally, by unifying (\ref{EQ:DgCgAifx=0}), (\ref{EQ:DgCgAif0<x<1}), and (\ref{EQ:DgCgAifx=1}), we obtain 
\[
D|(C|A)=
[ACD+(x+\mu(1-x))\no{A}D+\mu(1-x)\no{A}\no{D}]|[AC\vee (x>0)\no{A}].
\]
\end{proof}
\begin{remark}
As shown in Theorem \ref{THM:DgCgADynamic},  in general $D|(C|A)\neq D|AC$; that is,  the import-export principle is invalid also for antecedent-nested conditionals\footnote{This result answers to a specific question that G. Coletti posed to N. Pfeifer during the conference ``Reasoning under partial knowledge'' (Perugia, Italy,  December 14--15, 2018) in  honor of her 70\textsuperscript{th} birthday. }. Indeed,  $D|(C|A)=D|AC$ only when $x=0$ (in this special case the import-export principle holds) and  $D|(C|A)=D|(AC\vee \no{A})$ only when $x=1$.
In addition, $AC\subseteq C|A \subseteq AC \vee \no{A}$ (logical inclusion  among conditional events), which in numerical terms implies
 $AC\leq C|A \leq AC \vee \no{A}$. However, 
  there are no order relations among the three objects $D|AC$, $D|(C|A)$, and $D|(AC\vee \no{A})$ (see Table \ref{TAB:DgCgA}). 
	\begin{table}[!h]
	\centering
	\begin{tabular}{l|c|lcc|cc}		
		Constituent& $D|AC$& & $D|(C|A)$ & & $D|(AC\vee \no{A})$ \\ 
		&& $x=0$ & $0<x<1$  & $x=1$  &\\
		\hline
		$ACD$           & $1$ & $1$ \,        &            $1$   & $1$ &$1$\\
		$AC\no{D}$       &$0$ & $0$                    & $0$  &$0$ &$0$\\
		$A\no{C}$  &$t$  & $\mu$                  & $\mu$  &$\mu$ &$\nu$ \\
		$\no{A}D$  &$t$  & $\mu$                  & $x+\mu(1-x)$   &$1$   &$1$ \\
		$\no{A}\no{D}$   &$t$  &$\mu$                       & $\mu(1-x)$  &$0$  &$0$ 
	\end{tabular}
	\\[.5em]	
	\caption{$D|(C|A)$ coincides with  $D|AC$, when $x=0$.  $D|(C|A)$ coincides with  $D|(AC\vee \no{A})$, when $x=1$.  When $0<x<1$ there are no order relations among  $D|AC$, $D|(C|A)$, and $D|(AC\vee \no{A})$. \label{TAB:DgCgA}	}
\end{table}		
\end{remark}	
\begin{theorem}\label{THM:NEGDgCgA}
Let $A,C,D$ be three  events, with $AC\neq \emptyset$. Then, $\no{D}|(C|A)=1-D|(C|A)$.
\end{theorem}
\begin{proof}
We distinguish two cases: case $(i)$ $x>0$ and case $(ii)$ $x=0$.\\
Case $(i)$. We set $P(C|A)=x$, $\prev[D\wedge(C|A)]=z$, and  $\prev[D|(C|A)]=\mu$. Then, we recall that by the linearity of prevision   $\prev[D|(C|A)]=\prev[D\wedge (C|A)]+\mu \prev[1-C|A]$, that is 
$z=\mu x$.
By setting $\prev[\no{D}|(C|A)]=\eta$ and  by applying (\ref{EQ:SHORTITER-COND}) to the iterated conditional $(\no{D}|\Omega)|(C|A)=\no{D}|(C|A)$, we obtain  
$\no{D}|(C|A)=\no{D}\wedge (C|A)+\eta(1-C|A)$.
Then, as $C|A=D\wedge(C|A)+\no{D}\wedge(C|A)$ (see (Proposition 1 in \cite{ECSQARU17}), it holds that
	\begin{equation}\label{EQ:noDgCgA}
	\no{D}|(C|A)= C|A -D\wedge (C|A)+\eta(1-C|A)=\left\{\begin{array}{ll}
	0, &\mbox{ if }  ACD \mbox{ is true,}\\
	1, &\mbox{ if }  AC\no{D} \mbox{ is true,}\\
	\eta, &\mbox{ if }  A\no{C}\mbox{ is true,}\\
	\eta(1-x), &\mbox{ if }  \no{A}D \mbox{ is true,}\\
	x+\eta(1-x), &\mbox{ if }  \no{A}\no{D} \mbox{ is true.}
	\end{array}
	\right.
	\end{equation}
By the linearity of prevision it holds that  $\prev[\no{D}|(C|A)]=P(C|A)-\prev[D\wedge (C|A)]+\eta \prev[1-C|A]$, that is  $\eta=x-z+\eta(1-x)$, shortly: $z=x-\eta x$. 
From $z=\mu  x$ and $z=x-\eta x$ it follows that $\mu x=x-\eta x$, that is $x=(\mu+\eta)x$. Then, as $x>0$, it follows that $\eta=1-\mu$. In Table \ref{TAB:DnDgCgA} we show that  $D|(C|A)+ \no{D}|(C|A)$ is constant and coincides with 1 in all possible cases.
	\begin{table}[!ht]
		\centering\footnotesize
		\begin{tabular}{l|c|c|c}		
						 Constituent&  $D|(C|A)$ & $\no{D}|(C|A)$ & $D|(C|A)+ \no{D}|(C|A)$ \\ 
			\hline
			$ACD$           & $1$    &   $0$   & $1$\\
			$AC\no{D}$      & $0$     &   $1$  & $1$\\
			$A\no{C}$     & $\mu$     &   $1-\mu$  & $1$\\
			$\no{A}D$  & $x+\mu-\mu x$  & $(1-\mu)(1-x)=1-\mu-x+\mu x$   &$1$  \\
			$\no{A}\no{D}$   & $\mu-\mu x$  & $x+(1-\mu)(1-x)=1-\mu+\mu x$   &$1$
		\end{tabular}
\\[.5em]	
		\caption{Numerical values of $D|(C|A)$, $\no{D}|(C|A)$, and $D|(C|A)+\no{D}|(C|A)$ in the case when $x=P(C|A)>0$  and $\mu=\prev[D|(C|A)]$. By coherence, $\prev[\no{D}|(C|A)]=1-\mu$. See also equations (\ref{EQ:DgCgA}) and (\ref{EQ:noDgCgA}).  \label{TAB:DnDgCgA}	}
	\end{table}	\\
Case $(ii)$. As $x=0$, from Theorem \ref{THM:DgCgADynamic} (see  equation (\ref{EQ:DgCgAifx=0})), it holds that $D|(C|A)=D|AC$. Likewise, it holds that $\no{D}|(C|A)=\no{D}|AC$. Therefore, $D|(C|A)+\no{D}|(C|A)=D|AC+\no{D}|AC=\Omega|AC=1$.
\end{proof}
We observe that in case of some logical dependencies among the events $A,C,D$ some constituents may become impossible, in which case some lines in Table \ref{TAB:DnDgCgA} disappear; but, of course, Theorem  \ref{THM:NEGDgCgA} still holds. 

\noindent Notice that, based on  Theorem \ref{THM:NEGDgCgA},  a natural notion of  negation  for $D|(C|A)$ is  given by $\no{D}|(C|A)$,
which  corresponds to the narrow-scope negation of  conditionals. We recall that also the (non-nested) \emph{conditional}  is traditionally negated by the narrow-scope negation of conditionals, i.e., the negation of $C|A$ is defined by $\no{C}|A=1-C|A$ (see Section \ref{sect:2.1}).
 \section{The iterated conditional $A|(C|A)$}
\label{SEC:AgCgA}
In this section we focus the analysis on the  iterated conditional $A|(C|A)$, which is a special case of $D|(C|A)$ when $D=A$.
After describing  $A|(C|A)$ as a conditional random quantity (with a dynamic conditioning event), we obtain the equality $\prev[A|(C|A)]=P(A)$, under the assumption $P(C|A)>0$. We also illustrate an urn experiment where such equality is natural. Then, we study some relations among $A|AC$, $A|(C|A)$, and $A|(AC\vee \no{A})$ by showing in particular that 	$A|(AC\vee \no{A})\leq   A|(C|A)\leq A|AC$. In addition, we give further results related with the equality $\prev[A|(C|A)]=P(A)$ and then we consider the Sue example.
\subsection{The iterated conditional $A|(C|A)$ and its prevision}
We recall that $A\wedge (C|A)=AC$, indeed
\begin{equation}\label{EQ:AandCgAisAC}
A\wedge (C|A)=	\left\{\begin{array}{ll}
1, &\mbox{ if }  AC \mbox{ is true,}\\
0, &\mbox{ if }  A\no{C}\mbox{ is true,}\\
0, &\mbox{ if }  \no{A} \mbox{ is true;}
\end{array}
\right.=	\left\{\begin{array}{ll}
1, &\mbox{ if }  AC \mbox{ is true,}\\
0, &\mbox{ if } \no{A}\vee \no{C}\mbox{ is true;}
\end{array}
\right.=AC.
\end{equation}
 By setting $\mu=\prev[A|(C|A)]$ and $x=P(C|A)$, from (\ref{EQ:SHORTITER-COND}) and (\ref{EQ:AandCgAisAC}) we obtain
\begin{equation}\label{EQ:AgCgA}
\small
A|(C|A)=A\wedge (C|A)+\mu(1-C|A)=AC+\mu(1-C|A)=	\left\{\begin{array}{ll}
1, &\mbox{ if }  AC \mbox{ is true,}\\
\mu(1-x), &\mbox{ if }  \no{A} \mbox{ is true,}\\
\mu, &\mbox{ if }  A\no{C}\mbox{ is true.}\\
\end{array}
\right.
\end{equation}
Then, by applying Theorem \ref{THM:DgCgADynamic} with $D=A$,  we obtain
\begin{corollary}\label{COR:AgCgADynamic}
	Let  a coherent assessment  $(x,\mu)$ on $\{C|A,A|(C|A)\}$ be given.
	The iterated conditional $A|(C|A)$ is  the  conditional random quantity 
	\begin{equation}\label{EQ:AgCgADynamic}
	\begin{array}{lll}
	A|(C|A)&=&
	[AC+\mu(1-x)\no{A}]|[AC\vee (x>0)\no{A}]=\\ \\
	&=&\left\{
	\begin{array}{ll}
	A|AC, & \mbox{ if } x=0,\\ 
	\,[AC+\mu(1-x)\no{A}]|(AC\vee \no{A}), & \mbox{ if } 0<x<1,\\
	A|(AC\vee\no{A}),& \mbox{ if } x=1.\\
	\end{array}
	\right.
	\end{array}
	\end{equation}	
\end{corollary}
We observe that $A|AC=1$ because $P(A|AC)=1$. The possible values of $A|AC$, $A|(C|A)$, and  $A|(AC\vee \no{A})$ are given in Table \ref{TAB:AgCgA}.
	\begin{table}[!h]
	\centering
	\begin{tabular}{c|c|lcc|cc}
		$C_h$ & $A|(AC\vee \no{A})$ &        &  $A|(C|A)$   &       & $A|AC$ &  \\
		&                     & $x=1$  &   $0<x<1$    & $x=0$ &        &  \\ \hline
		$AC$     &         $1$         & $1$ \, &     $1$      &  $1$  &  $1$   &  \\
		$\no{A}$   &         $0$         & $0$  &    $\mu(1-x)$   & $\mu$ & $1$  &  \\
		$A\no{C}$  &        $\nu$        & $\mu$  & $\mu$         & $\mu$  &  $1$   &
	\end{tabular}
	\\[.5em]	
	\caption{Possible values of $A|(AC\vee \no{A})$, $A|(C|A)$, and $A|AC$ relative to the constituents $C_h$. Here, $x=P(C|A)$,  $\nu=P[A|(AC\vee \no{A})]$, and
		$\mu=\prev[A|(C|A)]$. 
The iterated conditional $A|(C|A)$ coincides with $A|AC=1$ when $x=0$, and it coincides with $A|(AC\vee \no{A})$ when $x=1$.  \label{TAB:AgCgA}}
\end{table}	

The next result shows that  $\prev[A|(C|A)]=P(A)$,  when $P(C|A)>0$.
\begin{theorem}\label{THM:pA=pAgCgA}
	Let $A$ and $C$ be two events  with  $AC\neq \emptyset$. 
	If $P(C|A)>0$, then $\prev[A|(C|A)]=P(A)$.
\end{theorem}
\begin{proof}
	By applying equation (\ref{EQ:PRODUCT}), with $(B|K)|(A|H)=A|(C|A)$, it holds that   
\begin{equation}	\label{EQ:PAandCgA}
\prev[A\wedge (C|A)]=\prev[A|(C|A)] P(C|A).
\end{equation}
Moreover, $A\wedge (C|A)=AC$ and hence 
\begin{equation}	\label{EQ:PAC}
\prev[A\wedge (C|A)]=P(AC)=P(A)P(C|A).
\end{equation}
Thus, from (\ref{EQ:PAandCgA}) and (\ref{EQ:PAC}), it follows that 
\begin{equation}\label{EQ:PAgCgA=PA}
\prev[A|(C|A)] P(C|A)=P(A)P(C|A),
\end{equation}
from which it follows  $\prev[A|(C|A)]=P(A)$, when $P(C|A)>0$.
\end{proof}
Notice that  the two objects 
$A|(C|A)$ and $A$ are generally not equivalent (even if  $\prev[A|(C|A)]=P(A)$ when  $P(C|A)>0$)\footnote{Under the material conditional interpretation, where  the conditional \emph{if $\bigcirc$ then  $\Box$} is looked at as the event $\no{\bigcirc}\vee \Box$,   the iterated conditional \emph{if (if $A$ then $C$) then $A$} coincides with $A$. Indeed, in this case, the iterated conditional \emph{if (if $A$ then $C$) then $A$} would coincide with $(\overline{\no{A}\vee C}) \vee A=A\no{C}\vee A=A$. 
 }. Indeed, $A$ is an event, while as shown  in (\ref{EQ:AgCgADynamic}) the iterated conditional $A|(C|A)$  is in general a conditional random quantity.
We also observe that,  when $P(C|A)=0$, formula (\ref{EQ:PAgCgA=PA}) becomes $0=0$, but
 in general  $\prev[A|(C|A)]\neq P(A)$ because from (\ref{EQ:AgCgADynamic}) one has $\prev[A|(C|A)]=1$ and usually $P(A)< 1$. Thus, when $P(C|A)=0$, the equality $\prev[A|(C|A)]=P(A)$ holds only if $P(A)=1$. We also observe that the assessment $(y,\mu)$ on $\{A, A|(C|A)\}$ is coherent for every $(y,\mu)\in[0,1]^2$, while $y=\mu$ under the constraint $P(C|A)>0$. 
\begin{remark}\label{REM:URNS}
The equality  $\prev[A|(C|A)] = P(A)$, when $P(C|A)>0$, 
appears natural in many  examples. Recall,
for instance, 
the experiment  where a ball is drawn from an urn of unknown composition. If we consider the events $A=$ ``the urn contains 9 white balls and 1 black ball", $C=$ ``the  (drawn) ball is white", normally we  evaluate $P(C|A)=0.9$.
 Moreover, the degree of belief in the hypothesis $A$ seems to be completely ``uncorrelated'' with  the conditional \emph{if $A$ then $C$}.  Then, the equality
$\prev[A|(C|A)]=P(A)$ is reasonable,	whatever value we specify for $P(A)$. 	The same happens if, for instance, $A=$ ``the urn contains 2 white balls and 8 black balls", in which case $P(C|A)=0.2$. 
We notice that on the one hand, even if the equality $\prev[A|(C|A)]=P(A)$ may appear sometimes counterintuitive, it is a result of our theory, where the new objects of conjoined  and iterated conditionals are introduced. On the other hand, we do not see any motivations why  the conditional \emph{if $A$ then $C$} should modify the degree of belief in $A$. We will examine later (see  Section \ref{SEC:AgCgAandC} and  Section \ref{SEC:FURTHERITER}) some examples, where  $A|(C|A)$ is replaced by a suitable generalized iterated conditional (which makes explicit some latent information); then, as a consequence, the
 previous  equality will be replaced by an inequality. For instance, in Section \ref{SEC:AgCgAandC}, the latent information  will be made explicit by replacing $A|(C|A)$ by $A|((C|A)\wedge C)$; then, it will be shown that 
 $\prev[A|((C|A)\wedge C)]\geq P(A)$.
\end{remark}

\subsection{Some relations  among $A|AC$, $A|(C|A)$, and $A|(AC\vee \no{A})$}
In this subsection we examine  an order relation  among $A|AC$, $A|(C|A)$ and $A|(AC\vee \no{A})$.
\begin{theorem}\label{THM:AgCgAINEQ}
	Let $A$ and $C$ be two  events, with $AC\neq \emptyset$.  Then, the following order relation holds
	\begin{equation}\label{EQ:AgCgAINEQ}
	A|(AC\vee \no{A})\leq   A|(C|A)\leq A|AC=1.
	\end{equation}
\end{theorem}
\begin{proof}
	We set $P(C|A)=x$,  $P(A|(AC\vee \no{A}))=\nu$, and $\prev(A|(C|A))=\mu$.
	We distinguish three cases: $(i)$ $x=0$; $(ii)$ $0<x<1$; $(iii)$ $x=1$. 
\\
Case $(i)$. As $x=0$, from (\ref{EQ:AgCgADynamic}) $A|(C|A)=A|AC=1$. Then, the inequalities in (\ref{EQ:AgCgAINEQ}) are satisfied.\\
Case $(ii)$. As $0<x<1$, from (\ref{EQ:AgCgADynamic}) it holds that $A|(C|A)=[AC+\mu(1-x)\no{A}]|(AC\vee \no{A})$. Then, by observing that $A|(AC\vee \no{A})=AC|(AC\vee \no{A})$ and that $\mu(1-x)\no{A}\geq 0$, it follows 
\begin{equation}\label{EQ:AgCgAdiff}
\begin{array}{ll}
A|(C|A)-A|(AC\vee \no{A})=[AC+\mu(1-x)\no{A}]|(AC\vee \no{A})-AC|(AC\vee \no{A})=\\
=[AC+\mu(1-x)\no{A}-AC]|(AC\vee \no{A})=[\mu(1-x)\no{A}]|(AC\vee \no{A})\geq 0.
\end{array}
\end{equation}
  Then, the inequalities in (\ref{EQ:AgCgAINEQ}) are satisfied. \\
Case $(iii)$. As $x=1$, from  (\ref{EQ:AgCgADynamic}) it holds that $A|(C|A)=A|(AC\vee \no{A})$. Then, the inequalities in (\ref{EQ:AgCgAINEQ}) are satisfied.
\end{proof}	
We remark that from (\ref{EQ:AgCgAINEQ}) it follows that
	\begin{equation}\label{EQ:PrevAgCgAINEQ}
	P[A|(AC\vee \no{A})] \leq  \prev[A|(C|A)]\leq P(A|AC)=1.
\end{equation}
\begin{remark} 	
We recall that 
$   AC \leq C|A \leq AC \vee \no{A}$. Then, 
symmetrically (see Theorem \ref{THM:AgCgAINEQ}), 
it holds that
\[
A|AC \geq A|(C|A) \geq A|(AC \vee \no{A}).
\]
In other words, the iterated conditional $A|(C|A)$ is an intermediate object between $A|AC$ (which is obtained when in the iterated conditional we  replace the antecedent $C|A$ by $AC$) and $A|(AC \vee \no{A})$ (which is obtained  when we replace $C|A$ by the associated material conditional $AC \vee \no{A}$).	
\end{remark} 

In the next result we illustrate the relation among $P(C|A)$, $	P[A|(AC\vee \no{A})]$, and $\prev[A|(C|A)]$.
\begin{theorem}\label{THM:AgCgAdiff}
	Let $A$ and $C$ be two  events, with $AC\neq \emptyset$. We set $P(C|A)=x$, $	P[A|(AC\vee \no{A})]=\nu$, and $\prev[A|(C|A)]=\mu$. Then,
\begin{equation}\label{EQ:munux}
\mu=\nu+\mu(1-x)(1-\nu).
\end{equation}
\end{theorem}
\begin{proof}
	We distinguish three cases: $(i)$ $x=0$; $(ii)$ $0<x<1$; $(iii)$ $x=1$. 
\\
Case $(i)$. As $x=0$,
formula  (\ref{EQ:munux}) becomes $\mu=\nu+\mu(1-\nu)$, which is satisfied because, from (\ref{EQ:AgCgADynamic}), $A|(C|A)=A|AC=1$ and hence  $\mu=1$. \\
Case $(ii)$. As $0<x<1$, from (\ref{EQ:AgCgADynamic}) it holds that $A|(C|A)=[AC+\mu(1-x)\no{A}]|(AC\vee \no{A})$. Then, as $AC|(AC\vee \no{A})=A|(AC\vee \no{A})$, it follows that
\begin{equation}\label{}
\begin{array}{ll}
A|(C|A)=[AC+\mu(1-x)\no{A}]|(AC\vee \no{A})=A|(AC\vee \no{A})+\mu(1-x)\no{A}|(AC\vee \no{A}).\\
\end{array}
\end{equation}
Then, 
\[
\mu=\prev[A|(C|A)]=
P[A|(AC\vee\no{A})]+\mu(1-x)P[\no{A}|(AC\vee \no{A})]=\nu+\mu(1-x)(1-\nu),
\]
that is (\ref{EQ:munux}).
 \\
Case $(iii)$. 
As $x=1$,
formula  (\ref{EQ:munux}) becomes $\mu=\nu$, which is satisfied because,  from (\ref{EQ:AgCgADynamic}), $A|(C|A)=A|(AC\vee \no{A})$.
\end{proof}

\begin{remark} By exploiting (\ref{EQ:munux}), we can illustrate  another way  to establish  the result of Theorem \ref{THM:pA=pAgCgA}, that is 	$\prev[A|(C|A)]=P(A)$, when $P(C|A)>0$. Indeed,
(as $P(C|A)=x>0$) from (\ref{EQ:munux}) we obtain $\mu=\frac{\nu}{\nu+x-x\nu}$.
Moreover, by setting $P(A)=y$, it holds that
\[
\nu=P(A|(AC\vee \no{A})=\frac{P(C|A)P(A)}{P(C|A)P(A)+P(\no{A})}=\frac{xy}{xy+1-y}.
\]
Then,
\[
\mu=\frac{\frac{xy }{xy+1-y}}{\frac{xy }{xy+1-y}+x-x\frac{xy }{xy+1-y}}=\frac{xy}{xy+x^2y+x-xy-x^2y}=\frac{xy}{x}=y,
\]
that is:  $\prev[A|(C|A)]=P(A)$.
\end{remark}
\subsection{Further results related to the equality $\prev[A|(C|A)]=P(A)$}
We now illustrate some formulas which are obtained from  the equality $\prev[A|(C|A)]=P(A)$ by replacing some events  by their negation. 
\begin{remark}	\label{REM:AgnoCgA}
	Based on  Theorem \ref{THM:pA=pAgCgA}, by replacing  the event
	$C$  by $\no{C}$ or $A$  by $\no{A}$ it holds that 
	\begin{enumerate}		
		\item $\prev[A|(\no{C}|A)]=P(A)$,  if $P(\no{C}|A)>0$ (where  $A\no{C}\neq \emptyset$);
		\item $\prev[\no{A}|(C|\no{A})]=P(\no{A})$,  if $P(C|\no{A})>0$ (where  $\no{A}C\neq \emptyset$); 
		\item $\prev[\no{A}|(\no{C}|\no{A})]=P(\no{A})$,  if $P(\no{C}|\no{A})>0$ (where  $\no{A}\no{C}\neq \emptyset$).
	\end{enumerate}
\end{remark}
We give below a further result  where the consequent $A$ in the iterated conditional $A|(C|A)$ is replaced by $\no{A}$.
\begin{theorem}\label{THM:pAgnCgA}
	Let  two  events $A$ and $C$ be given,   with $AC\neq \emptyset$.   Then,
	$\prev[\no{A}|(C|A)]=P(\no{A})$,  if $P(C|A)>0$.
\end{theorem}
\begin{proof}
By Theorem \ref{THM:NEGDgCgA}, $\no{A}|(C|A)=1-A|(C|A)$. Hence, $\prev[\no{A}|(C|A)]=1-\prev[A|(C|A)]$. If ${P(C|A)>0}$, by Theorem \ref{THM:pA=pAgCgA}, $\prev[A|(C|A)]=P(A)$. Therefore, $\prev[\no{A}|(C|A)]=1-P(A)=P(\no{A})$. 
\end{proof}	
\begin{remark}	
	Based on the proof of  Theorem \ref{THM:pAgnCgA}, and by Remark \ref{REM:AgnoCgA}, we obtain similar formulas by replacing the consequent  $A$  by $\no{A}$ (and hence  $\no{A}$ by $A$):
	\begin{enumerate}		
		\item $\prev[\no{A}|(\no{C}|A)]=P(\no{A})$,  if $P(\no{C}|A)>0$ (where $A\no{C}\neq \emptyset$); 		
		\item $\prev[A|(C|\no{A})]=P(A)$,  if $P(C|\no{A})>0$ (where $\no{A}C\neq \emptyset$);
		\item $\prev[A|(\no{C}|\no{A})]=P(A)$,  if $P(\no{C}|\no{A})>0$ (where $\no{A}\no{C}\neq \emptyset$). 
	\end{enumerate}
\end{remark}

\subsection{On the  Sue example}\label{SEC:sue}
We will now consider the Sue example, where the events  $A$ and $C$ are defined as follows: $A=$\emph{Sue passes the exam}, $C=$\emph{Sue goes on a skiing holiday}, $C|A=$\emph{If Sue passes the exam, then she  goes on a skiing holiday}. 
We recall that,  by Corollary  \ref{COR:AgCgADynamic}, when $P(C|A)=0$  it holds that  $A|(C|A)=A|AC=1$ and hence 
\[
\prev[A|(C|A)]=P(A|AC)=1\geq P(A),
\] with  $\prev[A|(C|A)]>P(A)$ when $P(A)<1$. On the contrary,  when $P(C|A)>0$, by Theorem \ref{THM:pA=pAgCgA} it follows that  $\prev[A|(C|A)]=P(A)$, i.e.,  the 
degree of belief in \emph{Sue passes the exam, given that  if Sue passes the exam, then she  goes on a skiing holiday} coincides with the probability that \emph{Sue passes the exam}.
The equality $\prev[A|(C|A)] = P(A)$, in the Sue example,   may seem counterintuitive (see the criticism in \cite{douven11b}); but, in our opinion this happens because some latent information is not considered.
In what follows we will examine further iterated conditionals, where the antecedent $(C|A)$ is replaced by a suitable conjunction containing the additional information that was latent in the context of the Sue example. The new antecedent, with the extra information, explains the intuition that the degree of belief in $A$ (given the new antecedent) should increase.
A first relevant case is obtained if we replace the antecedent $C|A$ by $(C|A)\wedge C$, that is if we consider the iterated conditional \emph{(Sue passes the exam), given that  (if Sue passes the exam, then she  goes on a skiing holiday) and (she  goes on a skiing holiday)}. In other words we consider the generalized  iterated conditional $A|( (C|A)\wedge C)$, which is examined in  the next section.
\section{The generalized iterated conditional $A|((C|A)\wedge C)$}
\label{SEC:AgCgAandC}
In this section we consider the generalized  iterated conditional $A|((C|A)\wedge C)$. The next result  shows that  $\prev[A|((C|A)\wedge C)]\geq P(A)$. 
\begin{theorem}\label{THM:AgCgAandC>A}
Given two events $A$ and  $C$, with $AC\neq \emptyset$, it holds that $\prev[A|((C|A)\wedge C)]\geq P(A)$.
\end{theorem}
\begin{proof}
We set $x=P(C|A)$ and $\mu=\prev[A|((C|A)\wedge C)]$. 
As $A\wedge (C|A )\wedge C=AC$,
 from Definition \ref{DEF:GENITER}
it holds that
\begin{equation}\label{EQ:AgCgAandC}
\begin{array}{ll}
A|((C|A)\wedge C)= A\wedge (C|A )\wedge C+\mu(1-(C|A)\wedge C)= \\=
AC+\mu(1-(C|A)\wedge C)
=\left\{\begin{array}{ll}
1, &\mbox{ if }  AC \mbox{ is true,}\\
\mu(1-x), &\mbox{ if }  \no{A}C\mbox{ is true,}\\
\mu, &\mbox{ if }  \no{C} \mbox{ is true.}\\
\end{array}
\right.
\end{array}
\end{equation}	
Then, by the linearity of prevision, it follows that $\mu\prev[((C|A)\wedge C)]=P(AC)$.
We  observe that 
\[
(C|A) \wedge C= \left\{
\begin{array}{ll}
1, &\mbox{ if }  AC \mbox{ is true,}\\
x, &\mbox{ if }  \no{A}C\mbox{ is true,}\\
0, &\mbox{ if }  \no{C} \mbox{ is true,}\\
\end{array}
\right.
\]
that is, $(C|A) \wedge C=AC+x\no{A}C$. Then,
\[
\prev[(C|A) \wedge C]=P(AC)+P(C|A)P(\no{A}C)=P(C|A)(P(A)+P(\no{A}C))=P(C|A)P(A\vee C), 
\]
with 
\[
\prev[(C|A) \wedge C]>0 \Longleftrightarrow P(C|A)>0 \mbox{ and } P(A\vee C)>0.
\]
If $\prev[(C|A)\wedge C]>0$, then
\begin{equation}\label{EQ:AgCgAandC>A}
	\mu=\prev[A|((C|A)\wedge C)]=\frac{P(AC)}{\prev[(C|A)\wedge C]}=\frac{P(AC)}{P(C|A)P(A\vee C)}=\frac{P(A)}{P(A\vee C)}=P(A|(A\vee C))
\geq P(A),
\end{equation}
because $P(A\vee C)\leq 1$. Equivalently, the inequality in (\ref{EQ:AgCgAandC>A})  also follows because $A\subseteq A|(A\vee C)$.\\
If $\prev[(C|A) \wedge C]=0$ we distinguish two cases: $(i)$ $P(C|A)=0$; $(ii) P(A\vee C)=0$.\\
In case $(i)$, from (\ref{EQ:AgCgAandC})
we obtain 
\begin{equation}
\small
A|((C|A)\wedge C)=	\left\{\begin{array}{ll}
1, &\mbox{ if }  AC \mbox{ is true,}\\
\mu, &\mbox{ if }    \no{A}\vee\no{C}  \mbox{ is true.}
\end{array}
\right.
\end{equation}
Then, by coherence, $\mu=1$ and $A|((C|A)\wedge C)=A|AC=1$. Thus $\mu=1\geq P(A)$.	
\\
In case $(ii)$  it holds that $P(A)=0$ and hence $\mu\geq P(A)$.
\\
In conclusion, $\prev[A|((C|A)\wedge C)]\geq P(A)$.
\end{proof}
\begin{remark}
We observe that usually the inequality 
$\prev[A|((C|A)\wedge C)]\geq  P(A)$ is strict,  with the equality   satisfied only in extreme cases. Indeed, 
as shown in the proof of Theorem \ref{THM:AgCgAandC>A}, when $\prev[(C|A)\wedge C]>0$, we obtain that  $\prev[A|((C|A)\wedge C)]=P(A)$ only if $P(A\vee C)=1$, or $P(A)=0$.
  When $\prev[(C|A)\wedge C]=0$, we obtain that  $\prev[A|((C|A)\wedge C)]=P(A)$ only if $P(C|A)=0$ and $P(A)=1$, or $P(A\vee C)=0$ and $\prev[A|((C|A)\wedge C)]=0$. 
\end{remark}
Theorem \ref{THM:AgCgAandC>A} can be used to explicate the intuition in the  example of Sue  (\cite{douven11b}), discussed  in Section \ref{SEC:sue}. The original intuition was that learning the conditional {\em if $A$ then $C$} should increase your degree of belief in $A$. However, in this example, 
we  learn
 {\em if $A$ then $C$} and we have the latent information $C$; then,
 we believe both the conditional {\em if $A$ then $C$} and the event $C$.
Theorem \ref{THM:AgCgAandC>A} shows that having these two beliefs can increase our belief in $A$. In formal terms, if we replace  the antecedent $C|A$ by $(C|A)\wedge C$, we reach the conclusion that the degree of belief in $A|((C|A)\wedge C)$ is greater than or equal to $P(A)$. Such reasoning can be seen as a form of an abductive inference (see, e.g., \cite{sep-abduction}) and it is also an instance of  \emph{Affirmation of the Consequent} (AC): from {\em if $A$ then $C$} and $C$ infer $A$. This argument form is not logically valid. In probability logic, however, it is probabilistically informative and not p-valid (see Section \ref{SECT:AC}).

The next result is similar to Theorem \ref{THM:AgCgAandC>A}, with $P(A)$ replaced by $P(A|C)$.
\begin{theorem}\label{THM:AgCgAandC>AgC}
	Given two events $A$ and $C$, with $AC\neq \emptyset$, it holds that $\prev[A|((C|A)\wedge C)]\geq P(A|C)$.
\end{theorem}
\begin{proof}
We  distinguish three cases: $(i)$ $P(C|A)P(A\vee C)>0$; $(ii)$ $P(C|A)=0$; $(iii)$ $P(C|A)>0$ and  $P(A\vee C)=0$.
\\
In case $(i)$, as shown in the proof of Theorem \ref{THM:AgCgAandC>A}, it holds that $\prev[A|((C|A)\wedge C)]=P(A|(A\vee C))\geq P(A|C)$, because $A|C \subseteq A|(A\vee C)$.\\
In case $(ii)$, 
as shown in the proof of Theorem \ref{THM:AgCgAandC>A}, it holds that $A|((C|A)\wedge C)=A|AC=1$ and hence  $\prev[A|((C|A)\wedge C)]=1\geq  P(A|C)$.\\
In case $(iii)$, as we obtain:
\[
A|C= \left\{
\begin{array}{ll}
1, &\mbox{ if }  AC \mbox{ is true,}\\
0, &\mbox{ if }  \no{A}C\mbox{ is true,}\\
\eta, &\mbox{ if }  \no{C} \mbox{ is true,}\\
\end{array}
\right. \mbox{ and }
\begin{array}{ll}
	A|((C|A)\wedge C)
	=\left\{\begin{array}{ll}
		1, &\mbox{ if }  AC \mbox{ is true,}\\
		\mu(1-x), &\mbox{ if }  \no{A}C\mbox{ is true,}\\
		\mu, &\mbox{ if }  \no{C} \mbox{ is true,}\\
	\end{array}
	\right.
\end{array}
\]
where $x=P(C|A)>0$.  Then, $A|C\leq  	A|((C|A)\wedge C)$ when $C$ is true and by (Theorem \ref{THM:EQ-CRQ} and) Remark \ref{REM:INEQ-CRQ}  it follows that $ \prev[A|((C|A)\wedge C)]\geq P(A|C) $.
\end{proof}
As shown by Theorem \ref{THM:AgCgAandC>AgC}, in the Sue example the probability of the conditional $A|C$, that is \emph{(Sue passes the exam), given that  (Sue  goes on a skiing holiday)}, can increase if we replace the antecedent $C$ by $(C|A)\wedge C$, that is if we replace 
\emph{(Sue  goes on a skiing holiday)} by 
\emph{(if Sue passes the exam, then she  goes on a skiing holiday) and 
 (Sue  goes on a skiing holiday)}.

\begin{remark}
	We observe that, in the particular case where  $P(C|A)=1$, it holds that
	\[
	\begin{array}{ll}
	A|((C|A)\wedge C)
	=
	\left\{\begin{array}{ll}
	1, &\mbox{ if }  AC \mbox{ is true,}\\
	0, &\mbox{ if }  \no{A}C\mbox{ is true,}\\
	\mu, &\mbox{ if }  \no{C} \mbox{ is true.}\\
	\end{array}
	\right.
	\end{array}
	\]
	Then, $A|((C|A)\wedge C)$ coincides with $A|C$, when $C$ is true and, by Theorem \ref{THM:EQ-CRQ}, it follows that  $\prev(A|((C|A)\wedge C))=P(A|C)$ and $A|((C|A)\wedge C)=A|C$. 
\end{remark}
\section{Lower and Upper Bounds on Affirmation of the Consequent}
\label{SECT:AC}
In the previous section we considered an instance of AC in Theorem \ref{THM:AgCgAandC>A} and the Sue example. In this section we will give a general probabilistic analysis of AC as an inference rule, where  the premise set is $\{C,C|A\}$ and the conclusion is $A$. 
We recall that a family of conditional events
 $\mathcal{F} = \{E_i|H_i \, , \; i=1,\ldots,n\}$  is  {\em p-consistent} if and only if
the assessment $(1,1,\ldots,1)$ on $\mathcal{F}$ is coherent. In addition,
a p-consistent family  $\mathcal{F}$ {\em p-entails} a conditional event $E|H$  if and only if the unique coherent extension on $E|H$ of the assessment $(1,1,\ldots,1)$ on $\mathcal{F}$ is  $P(E|H)=1$ (see, e.g., \cite{gilio13ijar}). 
We say that the inference from $\F$ to $E|H$ is \emph{p-valid} if and only if  $\mathcal{F}$  p-entails $E|H$. The characterization of the p-entailment using the notions of conjunction and  generalized iterated conditional has been given in \cite{GiPS18wp,GiSa19}.
We will show that the inference AC is not p-valid, that is  from $P(C)=P(C|A)=1$ it does not follow that $P(A)=1$.
 Notice that this inference rule has been also examined in \cite{pfeifer09b}. Here we examine the inference rule  without assuming that $0<P(C|A)<1$.
\begin{theorem} \label{THM:AC}
	Let  two logically independent events $A$ and $C$ and any coherent  probability assessment $(x,y)$ on $\{C,C|A\}$ be given. The extension $z=P(A)$ is coherent if and only if $z'\leq z\leq z''$, where
	\begin{equation}\label{EQ:AC}
	\begin{array}{ll}
	z'=0, \;\;&
	z''=
	\begin{cases}
	1, &\mbox{if } x=y,\\
	\frac{x}{y}, &\mbox{if } x<  y,\\
	\frac{1-x}{1-y}, &\mbox{if }    x>y.\\
	\end{cases}
	\end{array}
	\end{equation}
\end{theorem}	
\begin{proof}
	The constituents are $AC,A\no{C},\no{A}C,\no{A}\no{C}$; the associated points for the family $\{C,C|A,A\}$ are $Q_1=(1,1,1)$, $Q_2=(0,0,1)$, $Q_3=(1,y,0)$, $Q_4=(0,y,0)$. 
	We denote by $\mathcal{I}$  the convex hull of $Q_1,\ldots, Q_4$.	
	We observe that:  the assessment  $(x,y)$  on $\{C,C|A\}$ is coherent, for every $(x,y)\in[0,1]^2$;
	the assessment  $(x,z)$  on $\{C,A\}$ is coherent, for every $(x,z)\in[0,1]^2$; the assessment  $(y,z)$  on $\{C|A,A\}$ is coherent, for every $(y,z)\in[0,1]^2$. 
	Then,  the coherence of $(x,y,z)$ is equivalent to  the  condition $(x,y,z)\in \mathcal{I}$, that is to the
	solvability of the system
	\begin{equation}\label{EQ:SIGMA}
	\begin{array}{lllll}
	\left\{
	\begin{array}{lllllll}
	\lambda_1+\lambda_3=x, \\
	\lambda_1+y\lambda_3+y\lambda_4=y, \\
	\lambda_1+\lambda_2=z, \\
	\lambda_1+\lambda_2+\lambda_3+\lambda_4=1, \;\\
	\lambda_i\geq 0,\; i=1,\ldots,4\,,
	\end{array}
	\right. \mbox{ \;\; that is\;\; } 
	\left\{
	\begin{array}{lllllll}
	\lambda_1=yz,\\
	\lambda_2=z-yz,\\
	\lambda_3=x-yz,\\
	\lambda_4=1-x-z+yz, \;\\
	\lambda_i\geq 0,\; i=1,\ldots,4\,.
	\end{array}
	\right.
	\end{array}
	\end{equation}
	We observe that  the system $(\ref{EQ:SIGMA})$ with $z=0$ is solvable for every $(x,y)\in[0,1]^2$. Then, $z'=0$ for every $(x,y)\in[0,1]^2$. 
	Concerning the upper bound $z''$ we distinguish three cases: $(i)$ $x=y$;  $(ii)$ $x<y$;  $(iii)$ $x>y$. \\
	Case $(i)$. Given any $(x,y)$, with $x=y\in[0,1]$, the system $(\ref{EQ:SIGMA})$ is solvable with $z=1$. Then $z''=1$, when $x=y$.\\
	Case $(ii)$. Given any $(x,y)\in[0,1]^2$, with $x<y$, the system $(\ref{EQ:SIGMA})$ is solvable if and only if  $0\leq z\leq \frac{x}{y}$. Then $z''=\frac{x}{y}$, when $x<y$.\\
	Case $(iii)$. Given any $(x,y)\in[0,1]^2$, with $x>y$, the system $(\ref{EQ:SIGMA})$ is solvable if and only if  $0\leq z\leq \frac{1-x}{1-y}$. Then $z''=\frac{1-x}{1-y}$, when $x>y$. 
\end{proof}
As shown by Theorem~\ref{THM:AC}, the inference AC is not p-valid. Indeed, the assessment $(1,1,z)$ on $\{C,C|A,A\}$ is coherent for every $z\in[0,1]$. 
In \cite{GiPS18wp} it has been shown that
\[
\begin{array}{lll}
P(C)=1,\; P(C|A)=1,\; P(C|\no{A})<1\; \Longrightarrow\; P(A)=1.
\end{array}
\]
Thus, under the probabilistic constraint $P(C|\no{A})<1$, we obtain a (weak) p-valid  AC rule. 
\section{Further generalized iterated conditionals}
\label{SEC:FURTHERITER}
In this section we will examine some further cases, in the Sue example, where the degree of belief in $A$ increases if we consider suitable antecedents in the   generalized iterated conditionals.
\subsection{Analysis of  $A|((C|A) \wedge  K)$}
\label{SUBSEC:AgCgAandK}
We consider the case where the latent information is represented by the event  $K=$ \emph{Sue increases her study time}. 
As it seems reasonable, we assume that the following inequalities hold: 
\begin{equation}\label{EQ:ASSUMPTIONK}
(i)\; P(A|K)\geq  P(A);\;\; (ii)\;P(C|AK)\geq P(C|A);\;\; (iii)\;\prev[(C|A)\wedge K]>0.
\end{equation}
 Then, in order to take into account the  latent information $K$, we study  the generalized iterated conditional
$A|((C|A) \wedge  K))$.  
From Definition \ref{CONJUNCTION} it holds that 
$A\wedge (C|A)=AC$ and hence 
\begin{equation}\label{EQ:ACK}
A\wedge (C|A)\wedge K=ACK.
\end{equation}
We observe that $(C|A)\wedge K \wedge A$ is equal to $(C|A)\wedge K $, or $0$, according to whether $A$ is true, or false. Likewise, $(C|A)\wedge K \wedge \no{A}$ is equal to $(C|A)\wedge K $, or $0$, according to whether $\no{A}$ is true, or false. Therefore
$
(C|A)\wedge K=(C|A)\wedge K\wedge A+(C|A)\wedge K\wedge \no{A} 
$
and hence, from (\ref{EQ:ACK})
\begin{equation}\label{EQ:CgAandKDEC}
(C|A)\wedge K=ACK+(C|A)\wedge K\wedge \no{A}. 
\end{equation}
In addition, 
$\prev[(C|A)|\no{A}K]=P(C|A)$ because $A$ and $\no{A}K$ are incompatible (see \cite[formula (24)]{SPOG18}) and  from (\ref{EQ:COMPOUNDCN}) it follows that
\begin{equation}\label{EQ:CgAandKandnoA}
\prev[(C|A)\wedge  K\wedge \no{A}]=\prev[(C|A)|\no{A}K]P(\no{A}K)=P(C|A)P(\no{A}K).
\end{equation}
We observe that $\prev[(C|A)\wedge K]>0$ implies   $P(K)>0$ and $P(C|A)>0$.
Then, from (\ref{EQ:COMPOUNDCN}), (\ref{EQ:ACK}),
(\ref{EQ:CgAandKDEC}), and  (\ref{EQ:CgAandKandnoA}), we obtain
\begin{equation}
\begin{array}{ll}
\prev[A|((C|A) \wedge K)] = 
\frac{\prev[A\wedge (C|A) \wedge K]}{\prev[(C|A) \wedge K]}
=\frac{P(ACK)}{P(ACK)+P[(C|A)\wedge K \wedge \no{A}]}
=\frac{P(ACK)}{P(ACK)+P(C|A)P(\no{A}K)}=\\ =
\frac{P(AC|K)}{P(AC|K)+P(C|A)P(\no{A}|K)}
=P(A|K)\frac{P(C|AK)}{P(AC|K)+P(C|A)P(\no{A}|K)}
=P(A|K)\frac{P(C|AK)}{P(C|AK)P(A|K)+P(C|A)P(\no{A}|K)}.
\end{array}
\end{equation}
As $P(C|AK)\geq  P(C|A)>0$, it follows
\[
P(C|AK)=P(C|AK)P(A|K)+P(C|AK)P(\no{A}|K)\geq  P(C|AK)P(A|K)+P(C|A)P(\no{A}|K)>0,
\]
and hence
\[
\frac{P(C|AK)}{P(C|AK)P(A|K)+P(C|A)P(\no{A}|K)}\geq 1.
\]
Therefore, as $P(A|K)\geq  P(A)$, it holds that 
\begin{equation}
\begin{array}{ll}
\prev[A|((C|A) \wedge K)] =  P(A|K)\frac{P(C|AK)}{P(C|AK)P(A|K)+P(C|A)P(\no{A}|K)}\geq P(A|K)\geq P(A).
\end{array}
\end{equation}
In summary, if  the latent information is represented by the event $K$ and  
we assume the inequalities in  (\ref{EQ:ASSUMPTIONK}), then the degree of belief in $A$ increases, that is  $\prev[A|((C|A) \wedge K)] \geq P(A)$.
\begin{remark}
If we consider  the event $H$=\emph{Harry sees Sue buying some skiing equipment}, then (likewise (\ref{EQ:ASSUMPTIONK})) 
it is reasonable  to evaluate $P(A|H)\geq P(A)$,
$P(C|AH)\geq P(C|A)$, and 
$\prev[(C|A)\wedge H]>0$.
Then, by the  same reasoning, we reach the conclusion that 
$\prev[A|((C|A) \wedge H)] \geq P(A)$.
\end{remark}

\subsection{Analysis of the new object $A|[(C|A)\wedge (K|(C|A))]$}
It could seem  that the reasoning in  Section \ref{SUBSEC:AgCgAandK} should be developed by replacing  in $A|[(C|A)\wedge K]$ the event $K$ by the iterated conditional $K|(C|A)$, which formalizes the sentence \emph{if  Sue goes to holiday when she passes the exam, then she increases her study time}. 
If we replace $K$ by $K|(C|A)$, the iterated conditional  $A|[(C|A)\wedge K]$ becomes the new object $A|[(C|A)\wedge (K|(C|A))]$ which we need to examine.
We start our analysis by defining the conjunction  between a conditional event and  an iterated conditional.
\begin{definition}\label{DEF:CgF&BgKgAgH} Let  a conditional event $C|F$ and an iterated conditional $(B|K)|(A|H)$ be given, with $\prev[(B|K)|(A|H)]=\nu$. We define
\begin{equation}\label{EQ:CgF&BgKgAgH}
 (C|F)\wedge[(B|K)|(A|H)]=[(B|K)|(A|H)]\wedge (C|F)=
 (A|H)\wedge (B|K)\wedge (C|F)+\nu\,[(\no{A}|H)\wedge (C|F)].
\end{equation}
\end{definition}
We observe that, as  $(B|K)|(A|H)=(A|H)\wedge(B|K)+\nu\, \no{A}|H$, formula (\ref{EQ:CgF&BgKgAgH}) implicitly assumes a suitable distributive property. Indeed
\[
\begin{array}{l}
(C|F)\wedge[(B|K)|(A|H)]=(C|F)\wedge[(A|H)\wedge(B|K)+\nu\, \no{A}|H]=\\
=
 (A|H)\wedge (B|K)\wedge (C|F)+\nu\,[(\no{A}|H)\wedge (C|F)],
\end{array}
\]

and as we can see the equality 
\begin{equation}\label{EQ:DISTRIBUTIVE}
(C|F)\wedge[(A|H)\wedge(B|K)+\nu\, \no{A}|H]=
 (A|H)\wedge (B|K)\wedge (C|F)+\nu\,[(\no{A}|H)\wedge (C|F)]
\end{equation}
represents a kind of  distributive property of the conjunction over the sum. This property has  been already introduced in \cite{GiSa19PG}.
\begin{remark}
Notice that, 
by applying Definition \ref{DEF:CgF&BgKgAgH} with 
 $C|F=A|H$, as $(\no{A}|H)\wedge (A|H)=0$ it follows that 
 \begin{equation}\label{EQ:BgKgAgHgAgH}
[(B|K)|(A|H)]\wedge (A|H)=(A|H)\wedge(B|K), 
\end{equation}
which has the same structure of the equality $(B|A)\wedge A=AB$. Moreover, by recalling (\ref{EQ:PRODUCT}), it holds that  
\begin{equation}
\prev[((B|K)|(A|H))\wedge (A|H)]=\prev[(B|K)|(A|H)]P( A|H).
\end{equation}
\end{remark}
Concerning the antecedent $(C|A)\wedge (K|(C|A))$ of the new object studied in this section, $A|[(C|A)\wedge (K|(C|A))]$,
 as $K|\Omega=K$ from (\ref{EQ:BgKgAgHgAgH}) it follows that
\begin{equation}\label{EQ:(CgA&KgCgA}
(C|A)\wedge [K|(C|A)]=
(C|A)\wedge [(K|\Omega)|(C|A)]=(C|A)\wedge K.
\end{equation}
Thus  
\begin{equation}
A|[(C|A)\wedge (K|(C|A))] = A|[(C|A)\wedge K],
\end{equation}
that is the new object $A|[(C|A)\wedge (K|(C|A))]$ coincides with the generalized iterated conditional $A|[(C|A)\wedge K]$. Finally, for the degree of belief in $A|[(C|A)\wedge (K|(C|A))]$ we reach the same conclusions  given for $A|[(C|A)\wedge K]$ in  Section \ref{SUBSEC:AgCgAandK}.
\\
\subsection{The generalized iterated conditional $A|((C|A)\wedge (A|H))$ }
In this section we consider the generalized  iterated conditional where  in the antecedent we add the  conditional $A|H$. Here,  $H$ is a further event. Of course, when $H=\Omega$ it holds that 
\[
A|((C|A)\wedge (A|H))=A|((C|A)\wedge A)=A|AC=1\geq A.
\]
Hence, $\prev[A|((C|A)\wedge A)]=1\geq P(A)$. We show below that, given any event $H\neq \Omega$, we obtain the weaker result that $\prev[A|((C|A)\wedge (A|H))]=P(A\vee H)\geq P(A)$.
\begin{theorem}\label{THM:CgAandAgH=AvH}
	Let $A,C,H$ be three logically independent events, with $ P(ACH)>0$. Then, $\prev[A|((C|A)\wedge (A|H))]=P(A\vee H)$.
\end{theorem}
\begin{proof}
	We set $x=P(C|A)$, $y=P(A|H)$, $z=\prev[(C|A)\wedge (A|H)]$, $\mu=\prev[A|((C|A)\wedge (A|H))]$. We observe that 
	\begin{equation}\label{EQ:CgAandAgH}
	\begin{array}{l}
	(C|A)\wedge (A|H)=\left\{\begin{array}{ll}
	1, &\mbox{ if }  ACH \mbox{ is true,}\\
	y, &\mbox{ if }  AC\no{H} \mbox{ is true,}\\
	0, &\mbox{ if }  
	A\no{C}\vee \no{A}H
	\mbox{ is true,}\\
	z, &\mbox{ if }  \no{A} \no{H} \mbox{ is true.}\\
	\end{array}
	\right.
	\end{array}
	\end{equation}
	Thus, as $P(ACH)>0$, it holds that $ACH\neq \emptyset$ and hence
	$(C|A)\wedge (A|H)\neq 0$, so that $A|((C|A)\wedge (A|H))$ makes sense.  We also observe that $P(ACH)>0$
	implies $P(A\vee H)>0$; then 
	\begin{equation}\label{EQ:z>0}
	z=\prev[(C|A)\wedge (A|H)]=\frac{P(ACH)+yP(AC\no{H})}{P(A\vee H)}>0.
	\end{equation}
	Moreover,
	\[
	\begin{array}{l}
	A|((C|A)\wedge (A|H))=AC\wedge (A|H)+\mu(1-(C|A)\wedge (A|H))=\\
	=\left\{\begin{array}{ll}
	1, &\mbox{ if }  ACH \mbox{ is true,}\\
	y+\mu(1-y), &\mbox{ if }  AC\no{H} \mbox{ is true,}\\
	\mu(1-z), &\mbox{ if }  \no{A}\no{H} \mbox{ is true,}\\
	\mu, &\mbox{ if }  A\no{C} \vee \no{A}H \mbox{ is true.}\\
	\end{array}
	\right.\\
	\end{array}
	\]
	From  (\ref{EQ:COMPOUNDCN}) and (\ref{EQ:z>0})  it holds that 
	\[
	\mu=\prev[A|((C|A)\wedge (A|H))]=\frac{\prev[A\wedge (C|A)\wedge (A|H)]}{\prev[(C|A)\wedge (A|H)]}=\frac{\prev[AC\wedge (A|H)]}{\prev[(C|A)\wedge (A|H)]}.
	\]
	We observe that
	\[
	\begin{array}{l}
	(AC)\wedge (A|H)=\left\{\begin{array}{ll}
	1, &\mbox{ if }  ACH \mbox{ is true,}\\
	y, &\mbox{ if }  AC\no{H} \mbox{ is true,}\\
	0, &\mbox{ if }  \no{A}\vee \no{C} \mbox{ is true;}\\
	\end{array}
	\right.
	\end{array}
	\]
	then 
	\begin{equation}\label{EQ:ACandAgH}
	\prev[AC\wedge (A|H)]=P(ACH)+yP(AC\no{H}).
	\end{equation}
	Notice that (\ref{EQ:ACandAgH}) also follows by applying the distributive property; indeed \[
	AC\wedge(A|H)=AC\wedge(AH+y\no{H})=ACH+yAC\no{H},
	\]
	and hence (\ref{EQ:ACandAgH}) follows by the linearity of prevision.
	Finally, 
	\[
	\mu=\frac{\prev[AC\wedge (A|H)]}{\prev[(C|A)\wedge (A|H)]}=\frac{P(ACH)+yP(AC\no{H})}{\frac{P(ACH)+yP(AC\no{H})}{P(A\vee H)}}=P(A\vee H).
	\]
\end{proof}
Thus, under the assumption that $P(ACH)>0$, by Theorem \ref{THM:CgAandAgH=AvH}  it holds  that 
\begin{equation}\label{EQ:AgCgAandAgH>A}
\prev[A|((C|A)\wedge (A|H))]=P(A\vee H)\geq P(A).
\end{equation}
\begin{remark}
	If in  $A|((C|A)\wedge (A|H))$ we exchange the antecedent with the consequent, we obtain 
	the generalized iterated conditional 
	\begin{equation}\label{EQ:CgAandAgHgA}
	\begin{array}{l}
	((C|A)\wedge (A|H))|A=AC\wedge (A|H)+\eta(1-A)
	=\left\{\begin{array}{ll}
	1, &\mbox{ if }  ACH \mbox{ is true,}\\
	y, &\mbox{ if }  AC\no{H} \mbox{ is true,}\\
	0, &\mbox{ if }  A\no{C} \mbox{ is true,}\\
	\eta, &\mbox{ if }  \no{A} \mbox{ is true,}\\
	\end{array}
	\right.\\
	\end{array}
	\end{equation}
	where $\eta =\prev[((C|A)\wedge (A|H))|A]$. Notice that:\\ $(a)$ the conjunction $(C|A)\wedge (A|H)$ is a conditional random quantity with conditioning event $A\vee H$, that is $(C|A)\wedge (A|H)=X|(A\vee H)=(ACH+yAC\no{H})|(A\vee H)$ (see formula (\ref{EQ:CgAandAgH}));\\
	$(b)$ as $A\subseteq A\vee H$, it holds that (see Remark \ref{REM:AHK})
	\[
	((C|A)\wedge (A|H))|A=(X|(A\vee H))|A=X|A=(ACH+yAC\no{H})|A=(CH+yC\no{H})|A;
	\]
that is 
	the generalized iterated conditional $((C|A)\wedge (A|H))|A$ is a conditional random quantity with the conditioning event $A$;
	\\
	$(c)$ from (\ref{EQ:CgAandAgH}) and (\ref{EQ:CgAandAgHgA}) it holds that  $((C|A)\wedge (A|H))|A\geq (C|A)\wedge (A|H)$, when $A\vee H$ is true. \\
	Thus, by Remark \ref{REM:INEQ-CRQ} (and Theorem \ref{THM:EQ-CRQ}), it follows that 
	\begin{equation}\label{EQ:STRMOT}
	(i)\, \prev[((C|A)\wedge (A|H))|A]\geq \prev[(C|A)\wedge (A|H)];\;\;  (ii)\, ((C|A)\wedge (A|H))|A\geq (C|A)\wedge (A|H).
	\end{equation}
	Then, (under the assumption that ${P(ACH)>0}$) the inequality $\prev[A|((C|A)\wedge (A|H))]\geq P(A)$ in (\ref{EQ:AgCgAandAgH>A})    also follows by the following reasoning:
\begin{equation}{\label{EQ:INFIN}}
	\begin{array}{ll}
	\prev[A|((C|A)\wedge (A|H))]=\frac{\prev[AC \wedge (A|H)]}{\prev[(C|A) \wedge (A|H)]}=\frac{\prev[A\wedge (C|A) \wedge (A|H)]}{\prev[(C|A) \wedge (A|H)]}=P(A)\frac{\prev[(C|A) \wedge (A|H)|A]}{\prev[(C|A) \wedge (A|H)]}\geq P(A).
	\end{array}
\end{equation}
\end{remark}
Notice that the inequality $\prev[A|((C|A)\wedge (A|H))]\geq P(A)$ is valid for every event $H$; for instance, in 
the Sue example it is not necessary  that $H$ is the event \emph{Harry sees Sue buying some skiing equipment}. Indeed, the strong motivation for the inequality $\prev[A|((C|A)\wedge (A|H))]\geq P(A)$ in (\ref{EQ:INFIN}) is given by (\ref{EQ:STRMOT}).
\section{Independence and uncorrelation}
\label{SEC:INDINC}
In this section we give further comments on the notions of independence and uncorrelation.
As a first example, we recall that  $\prev[(C|A)\wedge A]=P(C|A)P(A)$ (see formula (\ref{EQ:PAC})). 
This result  is rejected by some researchers, because it is supposed to be ``deeply problematic''(\cite{douven11b}, see also \cite{edgington95}) for $A$ and $C|A$ to be  ``independent''. But,  we observe that for discussing  independence in the context of iterated conditionals a definition of independence should be  given first. 
We also observe that in our approach conditional events and conjunctions are conditional random quantities. In addition, concerning the notions  of independence and uncorrelation, 
we recall that two random quantities $X$ and $Y$ are \emph{uncorrelated} if  $\prev(XY)=\prev(X)\prev(Y)$.
While, $X,Y$  are \emph{stochastically  independent} when $P(X=x\wedge Y=y)=P(X=x)P(Y=y)$ for every pair $(x,y)$.
As it is well known, if  $X,Y$ are independent, then $X,Y$ are uncorrelated; the converse is in general not valid.
As independence implies uncorrelation, there are three cases: $X,Y$ are uncorrelated and independent; $X,Y$ are uncorrelated but not independent;  $X,Y$ are not uncorrelated and not independent.

A case where $X,Y$ are uncorrelated but not independent is the following one.   If $Y=X^2$ and $X\in\{-1,0,1\}$, with $P(X=1)=P(X=-1)=P(X=0)=1/3$, then $Y\in\{0,1\}$ with $P(Y=1)=2/3$ and $P(Y=0)=1/3$. In this case $\prev(X)=0$, $\prev(Y)=2/3$ and $\prev(X)\prev(Y)=0$. Moreover,  $XY=X^3=X$, then $\prev(XY)=0=\prev(X)\prev(Y)$. However,  $X,Y$ are not independent because, for instance, $P(X=1 \wedge Y=1)=P(X=1)=1/3$, while $P(X=1)P(Y=1)=2/9$, so that $P(X=1 \wedge Y=1)\neq P(X=1)P(Y=1)$.

From this point of view, the  equality  $\prev((C|A)\wedge A)=P(C|A)P(A)$  represents the property of  uncorrelation between the random quantities $C|A$ and $A$. Indeed, 
we observe that $ (C|A)\wedge A= (C|A)\cdot A$ in all cases and hence $\prev[(C|A)\wedge A]= P(C|A)P(A)=\prev[(C|A)\cdot A]$. Therefore, in our framework, we could look at $C|A$ and $A$ as uncorrelated random quantities. However,  $C|A$ and $A$ cannot be seen as two stochastically  independent random quantities because, for instance, by assuming $P(C|A)=P(A)=\frac{1}{2}$ ,   it holds that 
\[
P((A=1)\wedge ((C|A)=1))=P(AC=1)=P(AC)=\frac{1}{4}\neq P(A=1)P((C|A)=1)=P(A)P(AC)=\frac{1}{8}.\]
As a further example, 
let us consider  two conditional events $A|H$ and $B|K$, with $HK=\emptyset$. 
We recall that (see \cite[Section 5]{GiSa13a})
\[
\prev[(A|H)\wedge (B|K)]=\prev[(A|H)\cdot (B|K)]=P(A|H)P(B|K).
\]
Then, as already observed  in \cite{GiSa13a}, when $HK=\emptyset$ the conditional events $A|H$ and $B|K$ should be looked at as uncorrelated random quantities. However,  when $HK=\emptyset$, (the two random quantities)
$A|H$ and $B|K$ are (in general) not independent. Indeed, by setting $P(A|H)=x\in(0,1)$,   $P(B|K)=y\in(0,1)$, it holds that 
$A|H\in\{1,0,x\}$,  $B|K\in\{1,0,y\}$,   
$(A|H,B|K)\in\{(1,y),(0,y),(x,1),(x,0),(x,y)\}$, and
$(A|H)\cdot (B|K)\in\{0,x,y,xy\}$.
Then $A|H$ and $B|K$ are  not independent because,  by assuming for instance that $P(AH)>0$ and $P(\no{K})<1$, one has
\[
P[((A|H)=1)\wedge (B|K)=y]=P(AH\no{K})=P(AH)\neq  P(AH)P(\no{K})=P[(A|H)=1]P[(B|K)=y].
\]
\section{Conclusions}
\label{SEC:CONCLUSIONS}
In this paper, we examined Lewis's triviality proof, for what it could tell us about conditionals and iterations of them, and we studied the prevision of several iterated conditionals in the framework of coherence and conditional random quantities. 
We analyzed the antecedent-nested  conditional  $D|(C|A)$ and its negation $\no{D}|(C|A)$, by verifying  that the import-export principle does not hold. In particular,  we showed  that  $\prev[A|(C|A)]=P(A)$ and $\prev[\no{A}|(C|A)]=P(\no{A})$,  when $P(C|A)>0$. We also  proved the ordering $A|AC \leq A|(C|A) \leq A|(AC\vee \no{A})$.

Then, we examined the Sue example where the equality $\prev[A|(C|A)]=P(A)$ appears counterintuitive. 
To support the intuition that the degree of belief in  $A$ should increase, 
we introduced and studied the following  (generalized) iterated conditionals: $A|((C|A)\wedge C)$,  $A|((C|A)\wedge K)$,  $A|((C|A)\wedge (K|(C|A)))$, and 
$A|((C|A)\wedge (A|H))$. In these (generalized) iterated conditionals the respective antecedents are strengthened by additional information. This additional information can be seen as explicated latent information, which may derive   from background knowledge or  from conversational implicatures. We verified that for all these iterated conditionals, with suitably ``strengthened'' antecedents, the prevision is  greater than or equal to the probability of the consequent $A$. Thus these iterated conditionals seem  valid formalizations of different types of additional  information in the antecedent, for which it holds that  the degree of belief in $A$  increases.

Our examination of the Sue example illustrates for us a general point about the analysis of intuitions in philosophical 
thought experiments. If our intuitions are in conflict with the results of the available formal methods, it could be that 
$(i)$ the analysis requires a richer formal structure, or $(ii)$ implicit information in the thought experiment has to be made 
explicit for a correct analysis. The formal understanding of the Sue example requires both, a richer formal structure and 
the explication of implicit information.

We also deepened the study of the probabilistic propagation from the premises to the conclusion for the Affirmation of the Consequent, which is an abductive  inference form. Finally, we considered the equalities $\prev[(C|A)\wedge A]=P(C|A)P(A)$ and $\prev[(A|H)\wedge (B|K)]=P(A|H)P(B|K)$ when $HK=\emptyset$. 
Some authors see these equalities as cases of probabilistic independence, but we argued that such equalities are correctly interpreted (only) as instances of uncorrelation between two random quantities.

Lewis (\cite{lewis76}) was aware that, to retain the attractive qualities of the CPH and yet to avoid triviality, an approach something like ours would have to be developed. But he did not want himself to go down this line, because it would mean ``$\ldots$ too much of a fresh start $\ldots$'' and would burden ``$\ldots$ us with too much work to be done $\ldots$''. He specifically pointed to compounds and iterations of conditionals as serious problems for this fresh work. But we have done that work in this paper, proving that our approach implies the CPH as a natural consequence, avoids triviality, and leads to an account of compounds and iterations of conditionals.
\section*{Acknowledgments}
\noindent
We    thank    the    anonymous    reviewers    for    their    careful    reading    of    our    manuscript.  Their    many    insightful    comments    and    suggestions  were very helpful in improving this paper. Giuseppe Sanfilippo was partially supported by MIUR-FFABR 2017. Niki Pfeifer is supported by the German Federal Ministry of Education and Research, BMBF (grant 01UL1906X) and was supported by  Erasmus+.   
%\bibliographystyle{model2-names} % ordered by name
%\bibliography{ContributionSIColettiR2}

\end{document}